\documentclass[12pt]{amsart}
\usepackage{amsfonts,graphics,amsmath,amsthm,amsfonts,amscd,
  amssymb,amsmath,latexsym}
\usepackage{mathrsfs}
\usepackage{epsfig}
\usepackage{flafter,xspace}
\usepackage{datetime,chngcntr,xspace,comment,xcolor}
\usepackage[all,cmtip]{xy}\xyoption{dvips}
\usepackage{tikz-cd}
\usepackage[hidelinks]{hyperref}
\hypersetup{
    colorlinks,
    linkcolor={blue!50!black},
    citecolor={green!50!black},
    urlcolor={red!80!black}
}

\makeatletter

\renewcommand\subsection{
  \renewcommand{\sfdefault}{pag}
  \@startsection{subsection}%
  {2}{0pt}{-\baselineskip}{.2\baselineskip}{\raggedright
    \sffamily\itshape\small
  }}
\renewcommand\section{
  \renewcommand{\sfdefault}{phv}
  \@startsection{section} %
  {1}{0pt}{\baselineskip}{.2\baselineskip}{\centering
    \sffamily
    \scshape
}}

\makeatother

\theoremstyle{plain}
\newtheorem{Theorem}{Theorem}
\newtheorem{Corollary}[Theorem]{Corollary}
\newtheorem{theorem}{Theorem}[section]
\newtheorem{corollary}[theorem]{Corollary}

\newtheorem{lemma}[theorem]{Lemma}
\newtheorem{claim}[theorem]{Claim}
\newtheorem{proposition}[theorem]{Proposition}

\theoremstyle{definition}
\newtheorem{defn}[theorem]{Definition}

\newtheorem{definition-lemma}[theorem]{Definition-Lemma}

\theoremstyle{remark}
\newtheorem{remark}[theorem]{Remark}

\def\ol#1{\overline{#1}}

\def\ideal#1.{I_{#1}}
\def\ring#1.{\mathscr {O}_{#1}}
\def\fring#1.{\widehat{\mathscr {O}}_{#1}}
\def\proj#1.{\mathbb P(#1)}
\def\pr #1.{\mathbb P^{#1}}
\def\af #1.{\mathbb A^{#1}}
\def\Hz #1.{\mathbb F_{#1}}
\def\Hbz #1.{\overline{\mathbb F}_{#1}}
\def\pic#1.{\operatorname {Pic}\,(#1)}
\def\pico#1.{\operatorname{Pic}^0(#1)}
\def\picg#1.{\operatorname {Pic}^G(#1)}
\def\ner#1.{NS (#1)}
\def\rdown#1.{\llcorner#1\lrcorner}
\def\rup#1.{\ulcorner#1\urcorner}
\def\cone#1.{\operatorname {NE}(#1)}

\def\ccone#1.{\overline{\operatorname {NE}}(#1)}
\def\coef#1.{\frac{(#1-1)}{#1}}
\def\vit#1.{D_{\langle #1 \rangle}}
\def\mm#1.{\overline {M}_{0,#1}}
\def\H1#1.{H^1(#1,{\ring #1.})}
\def\ac#1.{\overline {\mathbb F}_{#1}}

\def\adj#1.{\frac {#1-1}{#1}}
\def\spn#1.{\overline{#1}}
\def\ses#1.#2.#3.{0\to #1\to #2\to #3 \to 0}
\def\pek#1.#2.{\Cal P^{#1}(#2)}
\def\plk#1.#2.{\Cal P^{\leq #1}(#2)}
\def\ev#1.{\operatorname{ev_{#1}}}
\def\bminv#1.{(\nu_1,s_1;\nu_2,s_2;\dots ;\nu_{#1},s_{#1};\nu_{r+1})}
\def\zinv#1.{(\nu_1,s_1;\nu_2,s_2;\dots ;\nu_{#1},s_{#1};0)}
\def\iinv#1.{(\nu_1,s_1;\nu_2,s_2;\dots ;\nu_{#1},s_{#1};\infty)}
\def\map#1.#2.{#1 \longrightarrow #2}
\def\rmap#1.#2.{#1 \dasharrow #2}
\def\emb#1.#2.{#1 \hookrightarrow #2}

\def\dim{\operatorname{dim}}

\def\deg{\operatorname{deg}}

\def\Spec{\operatorname{Spec}}
\def\chara{\operatorname{char}}

\def\mult{\operatorname{mult}}

\def\vol{\operatorname{vol}}

\def\N{\mathbb N}

\def\p{\mathbb P}

\def\e{\Cal E}

\def\e1{E_1}
\def\e2{E_2}

\def\OO{\mathscr O}
\def\Z{\mathbb Z}

\def\P{\mathbb P}
\def\CC{\mathscr C}
\newcommand\LL{\mathscr{L}}

\def\UU{\mathscr U}
\def\FF{\mathscr F}
\def\GG{\mathscr G}

\newcommand\Q{{\mathbb{Q}}}
\newcommand\R{{\mathbb{R}}}

\newcommand\slc{slc\xspace}
\newcommand\const{\upsilon}
\DeclareMathOperator{\coeff}{coeff}
\renewcommand\epsilon{\varepsilon}
\newcommand\wt{\widetilde}
\newcommand\wh{\widehat}
\newcommand\fra{\mathfrak a}
\newcommand{\sfint}{\mathsf{int}}

\title{On the boundedness of slc surfaces of general type}
\author{Christopher D. Hacon} 
\address{%
CDH:\newline\indent Department of Mathematics \\  
University of Utah\\  
Salt Lake City, UT 84112, USA}
\email{hacon@math.utah.edu}
\author{S\'andor J Kov\'acs}
\address{%
SJK:\newline\indent
University of Washington, Department of Mathematics, Seattle, WA
  98195-4350, USA} 
\email{skovacs@uw.edu}
\thanks{CDH was supported by NSF research grants no: DMS-1300750, DMS-1265285 and by
  a grant from the Simons Foundation; Award Number: 256202.}
\thanks{SJK was supported in part by NSF Grants DMS-1301888 and  DMS-1565352 and the
  Craig McKibben and Sarah Merner Endowed Professorship in Mathematics. }

\newcommand{\widepage}{
\textwidth= 6.5in
\textheight=8.75in
\voffset-.5in
\hoffset-.75in
\marginparwidth=56pt
}
\widepage
\begin{document}
\maketitle

The purpose of this note is to give a new proof of Alexeev's boundedness result for
stable surfaces which is independent of the base field and to highlight some
important consequences of this result.

Let $k$ be an algebraically closed field, an \emph{\slc model} $(X,B)$ is a
projective semi-log canonical pair
such that $K_X+B$ is ample (see Definition \ref{d-slc} below).  The main result of
this paper is the following.

\begin{Theorem}[Alexeev]\label{t-m}
  Fix a constant $\const>0$ and a DCC set $\CC \subset [0,1]\cap \Q$.  Then there
  exists an integer $r>0$ such that for any algebraically closed field $k$ and any
  two dimensional \slc model $(X,B)$ defined over $k$ with $\coeff(B)\subseteq \CC$
  and $(K_X+B)^2=\const$, $r(K_X+B)$ is very ample.
\end{Theorem}

\begin{remark}
  This result was originally proved by V. Alexeev in a series of papers
  (\cite{Alexeev89}, \cite{Alexeev93}, \cite{Alexeev94} and \cite{AM03}). The results
  there are stated for surfaces defined over a fixed algebraically closed field (of
  any characteristc) however, as pointed out to us by Alexeev, they actually hold
  independently of the field. We believe that this was known to some experts, however
  there are some subtleties in Alexeev's arguments that make the proof of the results
  over an arbitrary field not entirely routine.  In this paper we propose an
  alternative proof which we believe simplifies and makes more transparent Alexeev's
  original approach.  The main differences are: A substantial simplification of the
  arguments needed from \cite{Alexeev93}; The use the effective Matsusaka results of
  \cite{Ter99} and \cite{dCF15} (instead of the original papers of Matsusaka and
  Koll\'ar); The use of ultraproducts (cf.~\cite{Schoutens10} and \cite{BHMM11}) to
  simplify some of the arguments of \cite{Alexeev94}. Of course most arguments are
  heavily influenced by \cite{Alexeev94}.\end{remark}

\noindent
The following result, which is of independent interest, is a key step in the proof of
Theorem~\ref{t-m}.

\begin{Theorem}\label{t-dcc}
  Fix a DCC set $\mathscr C$. Let $\mathscr V=\{(K_X+B)^2\}$ where $(X,B)$ is a two
  dimensional \slc model defined over $k$, an algebraically closed field, with
  $\coeff(B)\subseteq \CC$. Then $\mathscr V$ is also a DCC set. In particular, there
  exists a number $\delta >0$, depending only on $\CC$, such that if $0<v\in \mathscr
  V$, then $v\geq \delta$.
\end{Theorem}

\begin{Corollary}\label{c-m} 
  Fix constants $\epsilon,\const>0$ and a DCC set $\CC \subset [0,1]\cap \Q$. Then
  the set of all two dimensional $\epsilon$-log canonical pairs $(X,B)$ defined over
  $k$ with $\coeff(B)\subseteq \CC$, $K_X+B$ nef and big and $(K_X+B)^2\leq \const$
  is degree bounded, i.e., there exists a constant $d>0$ such that for any pair
  $(X,B)$ as above there is a very ample divisor $H$ on $X$ such that $H^2\leq d$ and
  $B_{\rm red}\cdot H\leq d$.
\end{Corollary}

We have the following interesting applications which should allow the construction of
moduli spaces of (semi-log-canonical) canonically polarized surfaces for $p\gg 0$.


\begin{Theorem}
  \label{t-ssr} 
  Fix a constant $\const\in\N$ and a DCC set $\CC \subset [0,1]\cap\Q$. Then there
  exists a number $p_0>0$ such that if $L$ is an algebraically closed field of
  characteristic $p>p_0$, $(X,B)$ a pair defined over $L$ such that $\dim X=3$,
  $f:X\to S=\Spec L[[t]]$ a projective morphism with connected fibers
  such that, with $\eta\in S$ denoting the generic point of $S$,
  $\coeff(B_\eta)\subseteq \CC$, $(X_\eta, B_\eta)$ is semi-log canonical, and
  $K_{X_\eta}+ B_\eta$ is ample with $(K_{X_\eta}+ B_\eta)^2=\const$, then there
  exist a 
  separable finite morphism $S'\to S$, a projective morphism $f':X'\to S'$, and a
  pair $(X',B')$ such that $(X'_s,B'_s)$ is semi-log canonical and $K_{X'_s}+B'_s$ is
  ample for all $s\in S'$, and $(X'_{\eta '},B'_{\eta '})$ is isomorphic to $(X_{\eta
  },B_{\eta })\times _\eta \eta '$ where $\eta'\in S'$ denotes the generic point.
\end{Theorem}

\noindent
Theorem~\ref{t-ssr} will follow as Corollary~\ref{cor:stable-red-one} to the somewhat
more technical Theorem~\ref{t-ssr-two} which we only state later. It also implies
Corollary~\ref{cor:stable-red-two}, a variant of the above statement.

Finally, using Theorem~\ref{t-ssr} we will we prove another variant:


\begin{Theorem}\label{t-ssr-three}
  Fix a constant $\const\in \N$ and a DCC set $\mathscr C\subset [0,1]\cap \mathbb
  Q$. For each $m>0$ let $L_m$ be an algebraically closed field of characteristic
  $p_m>0$ such that $\lim p_m=\infty$ and let $k=[L_m]$.  Further let $(X_m,B_m)$ be
  a pair defined over $L_m$ such that $\dim X_m=3$, and let $f_m:X_m\to S_m$ be a
  projective morphism with connected fibers to a smooth curve.
  Assume that for each $m\in\Z$ with $\eta\in S_m$ denoting the generic point of
  $S_m$,
  $\coeff(B_{m,\eta})\subseteq \CC$, $(X_{m,\eta},
  B_{m,\eta})$ is semi-log canonical, and $K_{X_{m,\eta}}+ B_{m,\eta}$ is ample with $(K_{X_{m,\eta}}+
  B_{m,\eta})^2=\const$.

  Then 
  for all but finitely many $m$'s there exist a separable finite morphism $\sigma
  _m:S'_m\to S_m$, a projective morphism $X'_m\to S'_m$, and a pair $(X'_m,B'_m)$
  such that $(X'_{m,s},B'_{m,s})$ is semi-log canonical and $K_{X'_{m,s}}+B'_{m,s}$
  is ample for all $s\in S'_m$, and $(X'_{m,s},B'_{m,s})$ is isomorphic to
  $(X_{m,\sigma_m(s)},B_{m,\sigma_m(s)})$ for general $s\in S'_m$.
\end{Theorem}

\medskip

\noindent%
{\bf Acknowledgements.} %
The authors are grateful to Valery Alexeev for useful comments and especially for
explaining why the results of \cite{Alexeev94} hold independently of the field and to
Zsolt Patakfalvi for helpful remarks.

\section{Preliminaries}
\subsection{Definitions} 
We follow the definitions of \cite{SingBook} (in particular for discrepancies,
terminal, klt and lc pairs). A \emph{pair} $(X,B)$ consists of a demi-normal variety
$X$ (see Definition~\ref{d-slc}) and an effective $\Q$-divisor $B$ on $X$ such that
none of the irreducible components of $B$ is contained in ${\rm Sing}\,X$.  The set
of coefficients appearing in the irreducible decomposition $B=\sum_{i=1}^r b_iB_i$ is
denoted by $\coeff(B)=\{b_i\vert i=1,\dots,r\}$ and we let $B_{\rm red}=\sum _{i=1}^r
B_i$. Recall that a pair $(X,B)$ is \emph{$\epsilon$-klt} (resp.\
\emph{$\epsilon$-lc}) if $X$ is normal and $a(X,B)> \epsilon -1$ (resp.\ $a(X,B)\geq \epsilon -1$)
where $a(X,B)$ is the total discrepancy of $(X,B)$, in particular $b_i<1-\epsilon$
(resp.\ $b_i\leq 1-\epsilon$).

We say that $\CC\subset \R$ is a \emph{DCC set} if given any non-increasing sequence
$(a_i)_{i\in \N}$ of elements of $\CC$ then $(a_i)_{i\in \N}$ is constant for all
$i\gg 0$.  The typical example is $I=\{1-\frac 1 m|m\in \N\}$. We let $I_+=\{0\}\cup
\{i=\sum _{p=1}^l i_p|i_1,\ldots ,i_l\in I\}$ and $D(I)=\{a=\frac{m-1+f}{m},\ m\in
\N,\ f\in I_+\cap[0,1]\}$. It is well known that $I$ is a DCC set if and only if and
only if $D(I)$ is a DCC set.  Recall the following (see e.g., \cite[2.7]{Alexeev94}).

\begin{lemma}[Shokurov's Log Adjunction Formula]\label{l-sh} 
  Let $(X,S+B)$ be a log canonical surface pair where $B=\sum b_iB_i$ and $S$ is a
  prime divisor with normalization $\nu :S^\nu \to S$, then 
  \[(K _X+S+B)|_{S^\nu}=K_{S^\nu}+{\rm Diff}_{S^\nu}(B)=K_{S^\nu}+{\rm
    Diff}(0)_{S^\nu}+B|_{S^\nu}\] where the coefficients of ${\rm Diff}_{S^\nu}(B)$
  are $1$ or of the form $(n-1+\sum a_ib_i)/n\in [0,1]$ for some $n,a_i\in \N$. In
  particular, ${\rm coeff}({\rm Diff}_{S^\nu}(B))\subseteq D(I)$ if ${\rm
    coeff}(B)\subseteq I$.
\end{lemma} 

For later use we recall the following elementary observation.

\begin{lemma}\label{l-ineq} 
  Let $(X,S+B)$ be a log canonical surface pair where $S$ is a prime divisor with
  normalization $\nu :S^\nu \to S$. Then for any $1\geq \lambda\geq 0$ we
  have \[(K_X+S+\lambda B)|_{S^\nu}\geq K_{S^\nu}+\lambda {\rm Diff}_{S^\nu}(B).\]
\end{lemma}

\begin{proof} Let $B=\sum b_iB_i$. It suffices to show that $(n-1+\lambda \sum
  a_ib_i)/n\geq \lambda (n-1+\sum a_ib_i)/n$.
\end{proof}

For $a,b\in \R$ set $a\vee b={\rm max} \{a,b\}$ and $a\wedge b={\rm
  min}\{a,b\}$. Similarly, for $A=\sum a_iA_i$ and $A'=\sum a'_i A_i$ $\mathbb
R$-divisors set $A\vee A'=\sum (a_i\vee a'_i)A_i$ and $A\wedge A'=\sum (a_i\wedge
a'_i)A_i$.  A pair $(X,B)$ is a \emph{simple normal crossings pair} or an \emph{snc
  pair} if $X$ is smooth and the support of $B$ consists of smooth 
divisors meeting transversely.

\newcommand{\sncp}{snc pair\xspace}

\begin{defn}\label{d-slc} 
  A scheme $X$ is called \emph{demi-normal} if it is seminormal, $S_2$ and $G_1$ or
  equivalently if it is $S_2$ and its codimension $1$ points are either regular
  points or nodes (cf.\ \cite[5.1,10.14]{SingBook}). Let $X$ be a demi-normal scheme
  with normalization $\pi:\wt X\to X$ and conductors $D\subset X$ and $\wt D\subset
  \wt X$. Let $B\subset X$ be an effective $\Q$-divisor whose support does not
  contain any irreducible component of $D$ and $\wt B\subset \wt X$ the divisorial
  part of $\pi^{-1}(B)$.

  The pair $(X,B)$ is called \emph{semi-log canonical} or \emph{slc} if $X$ is
  demi-normal, $K_X+B$ is $\Q$-Cartier and $(\wt X, \wt B+\wt D)$ is log canonical.
  An \emph{\slc model} (or \emph{semi log canonical model}) is a projective pair
  $(X,B)$ which is slc and such that $K_X+B$ is ample.
\end{defn}

Let $\pi :X\to U$ be a projective morphism of normal varieties, then by definition $\pi _*\mathscr O _X(D)=\pi _*\mathscr O _X(\lfloor D\rfloor )$.
Given an $\mathbb R$-divisor $D$ on a normal projective variety $X$, the
\emph{volume} of $D$ is defined as \[{\rm vol}(D)=\lim _{m\to \infty} \frac
{h^0(\mathscr O _X(\lfloor mD \rfloor ))}{m^n/n!}.\] If $D$ is nef, then ${\rm
  vol}(D)=D^{\dim X}$. Note that ${\rm vol}(\lambda D)=\lambda ^{\dim X}{\rm vol}(D)$ for any $\lambda >0$.
It is easy to see that if $f:X\to Y$ is a morphism of normal
projective varieties, then ${\rm vol}(D)\leq {\rm vol}( f_* D)$ and if $E$ is an
$\mathbb R$-Cartier divisor on $Y$ such that $D-f^*E\geq 0$ is $f$-excepional, then
${\rm vol}(E)={\rm vol}(D)$.

Let $X$ be a quasi-projective variety then a \emph{b-divisor} $\mathbf B$ over $X$ is
given by a collection of divisors $\mathbf B_{X'}$ on $X'$ for any birational
morphism $X'\to X$ with the property that if $X''\to X$ is another birational
morphism and $\nu$ is a valuation corresponding to a divisor on $X'$ and $X''$, then
${\rm mult}_\nu (\mathbf B_{X'})={\rm mult}_\nu (\mathbf B_{X''})$. In other words a
$b$-divisor over $X$ is defined by its multiplicity along any divisor over
$X$. Similarly one defines $\mathbb R$-b-divisors etc.  

Let $(X,B)$ be a pair. Then typical examples of $\R$-b-divisors are as follows.

\begin{enumerate}
\item the \emph{discrepancy b-divisor}, $\mathbf A =\mathbf A_B$ is defined by the
  equation \[K_{X'}=\nu ^*(K_X+B) +\mathbf A_{B,X'}\] for any birational morphism
  $\nu :X'\to X$,
\item the b-divisor $\mathbf L=\mathbf L_B$ is defined by the
  equation \[K_{X'}+\mathbf L_{B,X'}=\nu ^*(K_X+B) +E_{B,X'}\] where $\mathbf
  L_{B,X'}$ and $E_{B,X'}$ are effective with no common components, for any
  birational morphism $\nu :X'\to X$,
\item the b-divisor $\mathbf M=\mathbf M_B$ is defined by $\mult _E(\mathbf M)=\mult
  _E(B)$ if $E$ is a divisor on $X$ and $\mult _E(\mathbf M)=1$ otherwise.
\end{enumerate} 

We have (cf.\ \cite[5.3]{HMX13}):

\begin{proposition}\label{p-vols} 
  Let $(X,B)$ be a projective \sncp, $f:Y\to X$ a log resolution of $(X,B)$, and
  $g:X\to Z$ a birational projective morphism such that $(Z,g_*B)$ is also an
  \sncp. Then
  \begin{enumerate}
  \item $\vol (K_X+B)=\vol (K_Y+\Gamma )$ for any $\mathbb R$-divisor $\Gamma $ such
    that $\Gamma -\mathbf L _{B,Y}\geq 0$ is $f$-exceptional.
  \item $\vol (K_X+B)=\vol (K_X+\Theta)$ where $\Theta =B\wedge \mathbf L_{g_*B,X}$.
  \end{enumerate}
\end{proposition}

\begin{defn}\label{d-bb} We say that a set of varieties $\mathscr X$ is {\it degree
    bounded} if there exists a constant $m>0$ such that for each $X\in \mathscr X$
  there is a very ample divisor $H$ on $X$ with $H^{{\rm dim}(X)}<m$. A set of pairs
  $\mathscr P$ is {\it degree bounded} if there exists an integer $m>0$ such that for
  each $(X,B)\in \mathscr P$ there is a very ample divisor $H$ on $X$ with $H^{{\rm
      dim}(X)}<m$ and $H^{{\rm dim}(X)-1}\cdot B_{\rm red}<m$. A set of pairs
  $\mathscr B$ is {\it log birationally degree bounded} if there exists a degree
  bounded set of pairs $\mathscr P$ such that for any $(X,B)\in \mathscr B$ there
  exists a pair $(Z,D)\in \mathscr X$ and a birational map $f:Z\dasharrow X$ such
  that $D_{\rm red }$ contains the strict transform of $B_{\rm red}$ and all
  $f$-exceptional divisors.
\end{defn}

\subsection{Ultraproducts}
We briefly recall a few results about ultrafilters and ultraproducts that will be
needed in what follows. The interested reader may consult \cite{Schoutens10} and
\cite{BHMM11} for more background.

We fix $\UU$ a \emph{non-principal ultrafilter} on $\N$ for the sequel. So $\UU$ is a
non-empty collection of infinite subsets of $\N$ such that
\begin{enumerate}
\item if $A\subset B\subset \N$ and $A\in \UU$, then $B\in \UU$,  
\item if $A,B\in \UU$, then $A\cap B\in \UU$, and
\item for any $A\subset \N$, either $A$ or $\N \setminus A$ are in
  $\UU$.\end{enumerate} We say that a property $P(m)$ holds for \emph{almost all}
$m\in \N$ if $\{m\in \N | P(m)\ {\rm holds }\}\in \UU$.  Let $(A_m)_{m\in \N}$ be a
sequence of rings then the \emph{ultraproduct} $[A_m]$ is defined by $[A_m]:=\left(
  \prod _{m\in \N}A_m\right)/\sim $ where $\sim$ is the equivalence relation defined
by $(a_m)\sim (b_m)$ iff $a_m=b_m$ for almost all $m\in \N$. $[a_m]\in [A_m]$ denotes the
equivalence class corresponding to the sequence $(a_m)_{m\in \N}$. Note that $a_m$
only needs to be defined for almost all $m\in \N$.  If, for almost all $m\in \N$, the $A_m$
are reduced (resp.\ fields), then so is $[A_m]$. Given a sequence of homomorphisms of
rings $f_m:A_m\to B_m$ then $[f_m]:[A_m]\to [B_m]$ is a homomorphism of rings defined
by $[f_m]([a_m])=[f_m(a_m)]$.

Suppose now that $L_m$ is a sequence of fields and $k=[L_m]$. For any fixed integer
$n>0$, we define the ring of \emph{internal polynomials} 
\[
k[x_1,\ldots,x_n]_{\rm int}=\left[ L_m[x_1,\ldots,x_n]\right].
\] 
Note that the name is misleading as the elements of $k[x_1,\ldots,x_n]_{\rm int}$ are
not necessarily polynomials.  
There exists a natural embedding 
\[
k[x_1,\ldots,x_n]\hookrightarrow k[x_1,\ldots,x_n]_{\rm int}
\] 
whose image is the set of elements $g=[g_m]\in k[x_1,\ldots,x_n]_{\rm int}$ of
bounded degree (i.e., such that there exists an integer $d$ with ${\rm deg}(g_m)\leq
d$ for almost all $m\in \N$).  For an ideal $\fra \subseteq k[x_1,...,x_n]$ we put
$\fra_{\rm int}:=\fra\cdot k[x_1,\dots,x_n]_{\rm int}$.

We have the following.

\begin{theorem}[\protect{\cite[Thm.~1.1]{MR739626}}]\label{l-ultra}
  The extension $k[x_1,\ldots,x_n]\hookrightarrow k[x_1,\ldots,x_n]_{\rm int}$ is
  faithfully flat. In particular, for an ideal $\fra \subseteq k[x_1,\dots,x_n]$,
  $\fra_{\rm int}\cap k[x_1,\dots,x_n] =\fra$.
\end{theorem}

This theorem implies that the ideals of $k[x_1,\dots,x_n]_{\rm int}$ generated in
bounded degree are in a one-to-one correspondence with the ideals of
$k[x_1,\dots,x_n]$, and hence they are all of the form $[\fra_m]$ for a sequence of
ideals $\fra_m\subseteq L_m[x_1,\dots,x_n]$, which are all generated in bounded
degree.

Given the fields $L_m$ for $m\in\N$ and assuming the above constructions, the symbol
$[X_m]$ stands for equivalence classes of sequences of schemes $X_m$ of finite type
over $L_m$ with respect to the equivalence relation: $[X_m]\sim [Y_m]$ iff $X_m=Y_m$
for almost all $m\in \N$. This $[X_m]$ is called an \emph{internal scheme} over $k$. An
\emph{internal morphism} $[f_m]$ is defined by (the equivalence class of) a sequence
of morphisms $f_m:X_m\to Y_m$ where as usual $[f_m]\sim[g_m]$ iff $f_m=g_m$ for
almost all $ m\in \N$. Similarly, an \emph{internal coherent sheaf} on an internal scheme
$[X_m]$, denoted by the symbol $[\FF_m]$, is defined as an equivalence class of
sequences of coherent sheaves $\FF_m$ on $X_m$ by the usual equivalence relation,
$[\FF_m]\sim[\GG_m]$ iff $\FF_m\simeq \GG_m$ for almost all $ m\in \N$, where $\FF_m$ and
$\GG_m$ are coherent sheaves on $X_m$ for almost all $m\in \N$.

\begin{claim}\label{constr:int-functor}
  There exists a functor $\sfint: X\mapsto X_{\rm int}$ from separated schemes of
  finite type over $k$ to internal schemes.
\end{claim}

\begin{proof}[Construction]
  First assume that $X$ is affine and define $\sfint$ as follows: let
  $X\hookrightarrow \mathbb A^N_k$ be a closed embedding defined by the ideal $\fra
  \subset k[x_1,...,x_N]$. As observed above $\fra_{\rm int} = [\fra_m]$ for an
  appropriate sequence of ideals, and then we define $X_m$ in $\mathbb A^N_{ L_m}$ by
  the ideal $\fra_m$ (it is enough to do this for almost all $m\in \N$).  Now we set
  $X_{\rm int}:=[X_m]$.  A similar construction applies to morphisms (for details see
  \cite[pp.~1468-1469]{BHMM11}) which implies that the above defined $X_{\rm int}$ is
  independent of the embedding we chose at the beginning and hence the construction
  is functorial.  By patching on an open cover of $X$ we obtain $X_{\rm int}$ in the
  general case and similarly the same for morphisms.
\end{proof}

Similarly to the functor $X\mapsto X_{\rm int}$ one may also define a functor
$\FF\mapsto \FF_{\rm int}$ from coherent sheaves on $X$ to internal coherent sheaves
on $X_{\rm int}$. The construction is relatively straightforward; for details and
basic properties see \cite[pp.~1471-1472]{BHMM11}.

Notice that since the construction of the functor $X\mapsto X_{\rm int}$ is based on
the defining ideal sheaf of $X$ and hence for a divisor $D\subset X$, the internal
subscheme $D_{\rm int}\subset X_{\rm int}$ is a divisor with corresponding divisorial
sheaf $\OO_{X_{\rm int}}(D_{\rm int})\simeq [\OO_{X_m}(D_m)]$. In other words, for a
divisor, the corresponding \emph{internal divisor} may be obtained either as an
internal scheme or an internal coherent sheaf. By \cite[3.5(i)]{BHMM11} Cartier
divisors correspond to Cartier divisors. An \emph{internal pair} $( X, D)$
consists of an internal scheme $ X$ and an internal $\Q$-divisor $ D\subset 
X$. For a pair $(X,D)$ over k, we will use the notation $(X,D)_{\rm int}:=(X_{\rm
  int},D_{\rm int})$.

Next, we establish an important connection
between bounding the degree of a projective variety and bounding the degree of its
defining polynomials. This is, of course, related to the fascinating Eisenbud-Goto
conjecture \cite{MR741934}, but we only need a much weaker statement, which we prove
below.

\begin{proposition}\label{prop:degree-to-degree}
  There exists a function $\beta:\N\times\N\to\N$, such that if $d,q,n\in\N$, $L$ is
  a 
  field and $X\subseteq \P^n$ is a projective variety over $L$ such that $\dim X=q$
  and $\deg X\leq d$, then $I(X)\subseteq L[x_0,\dots,x_n]$ can be generated by
  homogenous polynomials of degree at most $\beta(d,q)$.
\end{proposition}

\begin{proof}
  Note that if $\deg X\leq d$, then there exists a linear subspace $\P^{q+d-1}\simeq
  P\subseteq \P^n$ such that $X\subseteq P$ and hence we may assume that $n\leq
  q+d-1$. 
  We utilize an idea of Mumford: For any linear subspace $T\subseteq \P^n$ of
  dimension $n-q-2$ let $H_T$ denote the join of $X$ and $T$, i.e., the union of
  lines determined by pairs of points given by $X\times T$. Alternatively, $H_T$ may
  be defined as follows: consider the projection map $\pi_T:\P^n\dasharrow \P^{q+1}$
  and let $H_T:=\overline{\pi_T^{-1}{{\pi_T(X)}}}$.
  Clearly, $\deg H_T\leq \deg X\leq d$ and it is easy to see (cf.\ proof of
  \cite[Theorem~1,~p.32]{MR0282975}) that
  \[ 
  X \underset{\text{set-theoretically}}{=} 
  \bigcap_{T\cap X=\emptyset} H_T.
  \]
  It follows that there exists an ideal $J\subseteq L[x_0,\dots,x_n]$ that can be
  generated by homogenous polynomials of degree at most $d$ and such that $\sqrt J =
  I(X)$. 

  Now, the statement follows by \cite[Theorem~2.10(ii)]{MR739626}.
\end{proof}

\begin{proposition}\label{p-ultra} 
  Fix $d>0$ and let $L_m$ be a sequence of fields. Let $X_m$ be a sequence of
  projective varieties defined over $L_m$ of bounded degree and dimension for almost
  all $m\in \N$. Then there exists a projective variety $X$ defined over $k=[L_m]$
  such that $X_{\rm int}=[X_m]$.
  Furthermore, this $X$ admits an embedding to a projective space over $k$ such that
  its ideal sheaf is locally generated by polynomials of bounded degree.
\end{proposition}

\begin{proof} 
  Since the $X_m$ have bounded degree and dimension (for almost all $m\in \N$), we
  may assume that there are fixed integers $d,N>0$ such that for almost all $m\in
  \N$, $X_m$ is embedded in $\mathbb P ^N_{L_m}$ with degree $\leq d$.  Let $\mathbb
  A^N_k =: U\subset \mathbb P ^N_k$ be a standard open affine subset and let $U_{\rm
    int}=[U_m]$ be the corresponding internal scheme.  Consider $X_m\cap U_m\subseteq
  U_m\simeq\mathbb A^N_{L_m}$ with defining ideal $\fra_m$. It follows from
  Proposition~\ref{prop:degree-to-degree} that the $\fra_m$ are generated by
  polynomials of bounded degree and then so is $\fra_{\rm int}=[\fra_m]\subseteq
  k[x_1,\dots,x_N]_{\rm int}$. Let $\fra:=\fra_{\rm int}\cap k[x_1,\dots,x_N]$ and
  $X^U=Z(\fra)\subseteq \mathbb A^N_k=U$. By the construction we have that $[X_m\cap
  U_m]=(X^U)_{\rm int}$. Glueing the various $X^U$ as before we obtain the required
  closed subscheme $X\subseteq \P _k^N$.
\end{proof}

\begin{lemma}\label{l-u} 
  Let $X\subset \P^N_k$ where $k=[L_m]$ and $\LL$ a semiample line bundle on $X$. If
  $X_{\rm int}=[X_m]$ and $\LL_{\rm int}=[\LL _m]$, then $\LL _m$ is semiample for
  almost all $m\in \N$.
\end{lemma}

\begin{proof}
  Since $\LL$ is semiample, there is an integer $r>0$ such that $\LL^{\otimes r}$
  defines a morphism $\phi:X\to \p ^M_k$ with $\phi ^* \OO_{\p ^M_k}(1)\simeq
  \LL^{\otimes r}$. Since $(\p^M_k)_{\rm int}=[\p ^M_{L_m}]$ and $(\mathscr
  O_{\p^M_k}(1))_{\rm int}= [\OO _{\p ^M_{L_m}}(1)]$, it follows that $\LL
  _m^{\otimes r}= \phi _m^*(\OO _{\p ^M_{L_m}}(1))$ for almost all $m\in \N$ (where $(\phi
  )_{\rm int }=[\phi _m ])$ by \cite[3.5(iii)]{BHMM11}. Thus $\LL _m$ is semiample
  for almost all $m\in \N$.
\end{proof}

\begin{lemma}\label{l-um} 
  Let $(X,B)$ be a log canonical pair projective over $k=\bar k=[L_m]$ and let $\nu
  :X\dasharrow Y$ be a good minimal model (resp.\ log canonical model), then
  $(X_m,B_m)$ is a log canonical pair projective over $L_m$ and $\nu _m:X_m\dasharrow
  Y_m$ is a good minimal model (resp.\ log canonical model) for $(X_m,B_m)$ for
  almost all $m\in \N$, where $\nu _{\rm int}=[\nu _m]$, $Y_{\rm int}=[Y_m]$ and
  $B_{\rm int}=[B_m]$.
\end{lemma}

\begin{proof} 
  Since $r(K_X+B)$ is Cartier, so is $r(K_{X_m}+B_m)$ for almost all $m\in \N$ by
  \cite[3.5(i)]{BHMM11}. Let $\mu :X'\to X$ be a log resolution of $(X,B)$, then
  following \cite[Proof of 4.1]{BHMM11}, $\mu _m:X'_m\to X_m$ is a log resolution for
  almost all $m\in \N$ and $(K_{X'/X})_{\rm int}=[K_{X'_m/X_m}]$. Since the
  coefficients of $\mathbf A_{X'}=K_{X'}-\mu^*(K_X+B)$ are $\geq -1$, the same holds
  for the coefficients of $\mathbf A_{X'_m}=K_{X'_m}-\mu _m^*(K_{X_m}+B_m)$ for
  almost all $m\in \N$ (cf.\ \cite[\S 3]{BHMM11}).  We only discuss the case of good
  minimal models (the other case is very similar). By assumption $K_Y+B_Y$ is
  semiample where $B_Y=\nu _*B$. By Lemma \ref{l-u}, $K_{Y_m}+B_{Y_m}$ is semiample
  for almost all $m\in \N$. Let $p:W\to X$ and $q:W\to Y$ resolve $\nu$, then
  $p^*(K_X+B)\geq q^*(K_Y+B_Y)$ where the inequality is strict along
  $\nu$-exceptional divisors on $Y$.  But then, if $p_m:W_m\to X_m$ and $q_m:W_m\to
  Y_m$ are the corresponding morphisms, we have $p_m^*(K_{X_m}+B_m)\geq
  q_m^*(K_{Y_m}+B_{Y_m})$ where the inequality is strict along $\nu_m$-exceptional
  divisors on $Y_m$.
\end{proof}

\subsection{Effective Matsusaka and birational boundedness} 

We begin by recalling the following effective version of Matsusaka's theorem and a
vanishing theorem due to di Cerbo, Fanelli and Terakawa.

\begin{theorem}\label{T-van} Let $X$ be a smooth surface and $D$ be a big and nef
  Cartier divisor. Let $q_0:=(2{\rm vol}(K_X)+9)/(p-1)$, then $H^i(\mathscr O
  _X(K_X+qD))=0$ for all $i>0$ and $q> q_0$.
\end{theorem}

\begin{proof} 
  See \cite{Ter99} and \cite[5.7]{dCF15}.
\end{proof}

\begin{theorem}\label{dCFT} 
  Let $X$ be a smooth surface and $D$ be a big and nef Cartier divisor. Let $l=D^2-5$
  (resp.\ $l=D^2-9$) with $l\geq 0$ then if $|K_X+D|$ has a base point at $x\in X$
  (resp.\ $|K_X+D|$ does not separate points $x,y\in X$) then
  \begin{enumerate}
  \item If $X$ is not of general type nor quasi-elliptic of Kodaira dimension $1$,
    then there exists a curve $C\subset X$ containing $x$ (resp.\ containing one of
    the points $x,y$) such that $D\cdot C<2$ (resp.\ $D\cdot C<4$),
  \item If $X$ is of general type with $D^2>{\rm vol}(K_X)+6$ (resp.\ $D^2>{\rm
      vol}(K_X)+9$) or is quasi-elliptic of Kodaira dimension $1$ then there exists a
    curve $C\subset X$ containing $x$ (resp.\ containing one of the points $x,y$) such
    that $D\cdot C\leq 7$ (resp.\ $D\cdot C\leq 17$).
  \end{enumerate}
\end{theorem}
\begin{proof} 
  See the main theorem of \cite{Ter99} and \cite[4.9, 4.11]{dCF15}.
\end{proof}

\numberwithin{equation}{subsection}

\begin{corollary}\label{c-b} 
  Let $X$ be a normal surface and $D$ a nef and big Cartier divisor such that
  $D^2\geq {\rm vol}(K_X)$, then $|K_X+qD|$ defines a birational morphism for any
  $q\geq 18$.
\end{corollary}
\begin{proof} Let $\mu :X'\to X$ be a resolution and $D'=\mu ^* D$. Pick a general
  point $x\in X'$ (resp.\ general points $x,y\in X'$).  If $C$ is a curve on $X'$
  containing $x$ (resp.\ containing $x$ or $y$), then $qD'\cdot C\geq q\geq 18$ and
  $(qD')^2=(qD)^2\geq {\rm vol}(K_X) +10\geq {\rm vol}(K_{X'}) +10$. By Theorem \ref{dCFT}, $x$ is not a base
  point of $|K_{X'}+qD'|$ and $|K_{X'}+qD'|$ separates $x$ and $y$. Thus $|K_X+qD|$ defines a birational morphism.
\end{proof}

We will also need the following result which is analogous to \cite[3.1]{HMX13}.

\begin{theorem}\label{t-bb} 
  Fix $A\in \N$, $\delta >0$. Let $(X,B)$ be a log canonical surface such that the
  coefficients of $B$ are $\geq \delta$, ${\rm vol}(q(K_X+B))\leq A$ and
  $|K_X+q(K_X+B)|$ is birational for some $q>0$, 
  then $(X,B)$ is log birationally degree bounded.
\end{theorem}

\begin{proof} 
  The proof follows \cite[3.1]{HMX13}. For the convenience of the reader we include a
  sketch which highlights the main changes necessary to avoid the use of
  Kawamata-Viehweg vanishing.

  By a standard reduction (cf.~\cite[3.1]{HMX13}) we may assume that $X$ is smooth
  and $|K_X+q(K_X+B)|$ induces a morphism $\phi :X\to Z$, i.e.,
  $|K_X+q(K_X+B)|=|M|+E$ where $|M|$ is basepoint-free.  Let $H$ be a very ample
  Cartier divisor on $Z$ so that $M=\phi ^* H$.  It suffices to show that $H^2$ and
  $\phi _*B_{\rm red}\cdot H$ are bounded from above.

  Clearly \[H^2=M^2={\rm vol}(M)\leq {\rm vol}(K_X+q(K_X+B))\leq (q+1)^2{\rm
    vol}(K_X+B)\leq 2^2A.\] Let $D_0$ be the sum of the components of $B$ that are
  not $\phi$-exceptional.  Note that if $G\in |M|$, then there is an effective
  $\Q$-divisor $C=E+B-\delta D_0\geq 0$ such that
  \begin{equation}\label{e-v}
    \delta D_0+G+C\sim _{\Q}(q+1)(K_X+B).
  \end{equation}
  Let $\alpha = 2A+10$. Since $B\geq 0$ and $q>0$, 
  \[
  {\rm vol} (q(K_X+B))\geq {\rm vol}(K_X+B)\geq {\rm vol }(K_X)
  \] 
  and so $\alpha \geq 2{\rm vol}(K_X)+10$.  Then
  \begin{equation}
    \begin{gathered}
      \phi _*D_0\cdot H=D_0\cdot G\leq D_0\cdot 2\alpha G\leq 4{\rm
        vol}(K_X+D_0+2\alpha  G)\leq  \\  \leq
      4{\rm vol}\left(\left(1+\left(\frac 1 \delta
            +2\alpha\right)(q+1)\right)(K_X+B)\right)\leq 16 \left(1+\frac 1 \delta
        +2\alpha \right)^2 A.
    \end{gathered}
  \end{equation} 
  Here the first (in)equality follows as $G\sim M=\phi ^*H$, the second is trivial, 
  the third by Lemma \ref{l-in} below, the fourth by \eqref{e-v}, and
  the fifth since $\frac {1+(1/\delta +2\alpha)(q+1)}q\leq 2( 1+1/\delta +2\alpha )
  $.
\end{proof}

\begin{lemma}\label{l-in} 
  Let $X$ be a normal surface, $M$ a Cartier divisor such that $|M|$ is base point
  free and the induced map $\phi =\phi _{|M|}:X\to Z$ is birational. Let $L=2\alpha
  M$ for some integer $\alpha \geq {\rm vol}(K_X)+10$ and $D$ be a sum of distinct
  prime divisors, then
  \[D\cdot L \leq 4 {\rm vol}(K_X+D+L).\]
\end{lemma}


\begin{proof} 
  By standard reductions, we may assume that $(X,D)$ is an snc pair and the
  components of $D$ are disjoint and not $\phi$-exceptional (cf.~\cite[3.2]{HMX13}).
  Now consider the short exact sequence
  \begin{equation}\label{e-x}
    0\to \mathscr O _X(K_X+m L)\to  \mathscr O _X(K_X+m L+D)\to  \mathscr O _D(K_D+m
    L|_D)\to  0 
  \end{equation}
  Since $R^1\phi _* \mathscr O _X(K_X+mL)=0$ by the Grauert-Riemenschneider vanishing
  theorem (see \cite[10.4]{SingBook} for a version that applies here)
  using the projection formula and Serre vanishing, it follows that
  \[H^1 (\mathscr O _X(K_X+m L))=H^1 (\phi _* \mathscr O _X(K_X+m L))=H^1(\phi _*
  \mathscr O _X(K_X)\otimes \mathscr O _Z(mH))=0\] for all $m \gg 0$ and hence

  \begin{equation}\label{e-y}
    \begin{aligned} 
      H^0(\mathscr O _X(K_X+m L+D))\to H^0(\mathscr O _D(K_D+m
      L|_D))
    \end{aligned}
  \end{equation} 
  is surjective.  We now claim that 

  \begin{claim}\label{c-z}
    It follows that $ h^0(\mathscr O _X(K_X+L))>0$ and no component of $D$ is
    contained in the base locus of $|K_X+\alpha M+D|$.
  \end{claim}

  \begin{proof} Since $\alpha \geq 2{\rm vol}(K_X)+10$, by Theorem \ref{T-van},
    $H^1(\mathscr O _X(K_X+\alpha M))=0$ and hence
    \[H^0(\mathscr O _X(K_X+\alpha M+D))\to H^0(\mathscr O _D(K_D+\alpha
    M|_D))=\oplus H^0(\mathscr O _{D_i}(K_{D_i}+\alpha M|_{D_i}))\] is surjective
    where $D=\sum D_i$ and each $D_i$ is a prime divisor. Since the components of $D$
    are not $\phi$ exceptional, $M\cdot D_i>0$, $H^0(\mathscr O _{D_i}(K_{D_i}+\alpha
    M|_{D_i}))\ne 0$ for all $i$ and so a general element of $H^0(\mathscr O
    _X(K_X+\alpha M+D))$ does not vanish along any component of $D$.

    The proof that $ h^0(\mathscr O _X(K_X+L))>0$ is similar (and easier).
  \end{proof} 
  Now consider the following commutative diagram
  \[\CD \mathscr O_X(K_X+mL+D) @>>> \mathscr O _D(K_D+mL|_D)\\
  @VVV    @VVV\\
  \mathscr O_X((2m-1)(K_X+L+D)) @>>> \mathscr O _D((2m-1)(K_D+L|_D))
  \endCD
  \] 
  where the vertical maps are induced by a general divisor in
  $|(m-1)(2K_X+L+2D)|=|2(m-1)(K_X+\alpha M+D)|$. Since no component of $D$ is
  contained in the support of this divisor (Claim \ref{c-z}), it follows that
  \begin{equation}\label{e-xx}
    \begin{gathered}
      h^0(\mathscr
      O_X((2m-1)(K_X+L+D)))-h^0(\mathscr O_X((2m-2)(K_X+L+D)+K_X+L))= \\
      \dim {\rm Im}\left[ H^0(\mathscr O_X((2m-1)(K_X+L+D)))\to H^0(\mathscr O
        _D((2m-1)(K_D+L|_D)))\right]\geq \\
      h^0(\mathscr O _D(K_D+mL|_D)).
    \end{gathered}
  \end{equation} 

  Let $P(m)=h^0(\mathscr O_X(m(K_X+L+D)))$, then $P(m)=m^2 {\rm
    vol}(K_X+L+D)+o(m^2)$.  Since $h^0(\mathscr O _X(K_X+L))>0$ (see Claim
  \ref{c-z}), we have $h^0(\mathscr O _X((2m-2)(K_X+L+D)+K_X+L))\geq P(2m-2)$ and
  hence by \eqref{e-xx}, we have

  \begin{equation}\label{e-yy} 
    P(2m-1)-P(2m-2)\geq Q(m):=h^0(\mathscr O    _D(K_D+mL|_D))=mL\cdot D+o(m).
  \end{equation}

  Comparing leading terms of $P(m) $ and $Q(m)$, it follows that
  \[ 
  4{\rm vol}(K_X+L+D)\geq L\cdot D.\qedhere
  \]
\end{proof}


\section{The proofs of the main results}
\begin{subsection}{Preliminary results}

  \begin{lemma}\label{l-a} Fix $\mathscr C\subset [0,1]$ a DCC set, then there exists
    a constant $V >0$ such that if $(X,B)$ is a klt surface such that $\rho (X)=1$,
    ${\rm coeff}(B) \subset \CC$ and $K_X+B\sim _\Q 0$, then $(-K_X)^2\leq V$.
  \end{lemma}
  \begin{proof} Suppose that $(-K_X)^2>V$, then for any smooth point $x\in X$ there
    exists a $\Q$-divisor $G\sim _\Q -K_X$ such that ${\rm mult}_x (G)>V^{1/2}$
    (cf.~\cite[10.4.12]{Laz04}). Since $\rho (X)=1$, we may assume that all
    components of $G$ contain a general point $x\in X$ and in particular are not
    contained in the support of $B$. Let $\Phi =(1-\delta)B+\delta G$ such that
    $(X,\Phi )$ is log canonical but not klt. Notice that $0<\delta < 2/V^{1/2}$
    (cf.~\cite[9.3.2]{Laz04}). Perturbing $G$ we may in fact assume that there is a
    unique non-klt center $Z$ for $(X,\Phi)$.

    If $Z$ is a divisor, then (since $\rho (X)=1$) we may assume that $\delta
    G=Z$. Restricting to $Z$ we have \[0\equiv (K_X+(1-\delta )B+\delta G)|_Z=K_Z
    +{\rm Diff}_Z((1-\delta )B).\] Since $\deg {\rm Diff}_Z((1-\delta )B)\geq 0$ then
    ${\rm deg}(K_Z)\in \{0,-2\}$. If $B\neq 0$, then \[2=\deg {\rm Diff}_Z((1-\delta
    )B)\] (see Lemma \ref{l-sh}) easily implies that $\delta$ is bounded from below
    (cf.~\cite[5.2]{HMX14}) and hence $(-K_X)^2$ is bounded from above.  If $B=0$ then
    $K_X\equiv 0$ and the claim is trivial.

    Therefore we may assume that $\dim Z=0$. Let $\nu :X'\to X$ be the extraction of
    the corresponding curve $E$ of discrepancy $-1$ so that $K_ {X'}+E+\Phi '=\nu
    ^*(K_X+\Phi)\sim _\Q 0$ where $\Phi '=\nu ^{-1}_* \Phi$. Write $K_{X'}+B'+aE=\nu
    ^*(K_X+B)\sim _\Q 0$ where $a<1$.  We run the first step of the $K_{X'}+\Phi
    '\equiv -E $ minimal model program. If the induced rational map is a Mori fiber
    space $X'\to W$, then restricting to a general fiber $F$ we let $\Phi
    ''=\Phi'|_F$, $E''=E|_ F$ and $B''=B'|_F$, and if the induced rational map is a
    divisorial contraction $\pi :X'\to F$, then we let $\Phi ''$, $E''$ and $B''$ be
    the pushforwards of $\Phi'$, $E$ and $B'$.  We have that \[K_F+\Phi '' +
    E''\equiv 0, \qquad K_F+B'' +aE''\equiv 0,\] $K_F+B''+E'' \equiv  (1-a)E''$ is ample and since
    $\Phi ''\geq (1-\delta)B''$, then
  \[\delta B''\equiv (1-a)E''+\Phi ''-(1-\delta )B''\geq 0.\]  It follows that
  $B''\ne 0$ and $K_F+(1-\eta)B''+E''\equiv 0$ for some $0<\eta <\delta$.

  If $\dim F =1$, then since the coefficients of $B''$ are in the DCC set $\CC$,
  there exists a constant $\beta >0$ such that $\deg (\eta B'')=\deg
  (K_F+B''+E'')\geq \beta$. But then, since $K_F+B''+aE''\equiv 0$, we have
  \[
  2\geq \deg (B'')\geq \beta /\eta >\beta /\delta\geq \beta V^{1/2}/2
  \] 
  and so $V$ and hence $(-K_X)^2$ are bounded from above.

  If $\dim F=2$, let $(K_{F}+B''+E'')|_{E''}=K_{E''}+{\rm Diff}_{E''}(B'')$. Since
  $(F,E'')$ is purely log terminal, then $E''$ is smooth and by adjunction the
  coefficients of ${\rm Diff}_{E''}(B'')$ are in the DCC set $D(\CC )$ and so
  \[-2+\deg ({\rm Diff}_{E''}(B''))=(K_{F}+B''+E'')\cdot E''\geq \beta -2>0\] where
  $\beta ={\rm min}\{ \sum b_i'|b_i'\in D(\CC ),\ \sum b_i'>2\}$.  Fix $\lambda$ such
  that $K_{E''}+\lambda {\rm Diff}_{E''}(B'')\equiv 0$, then as $E''$ is
  rational, \[\lambda =\frac 2{{\rm deg}({\rm Diff}_{E''}(B''))}\leq \frac 2 \beta
  <1.\] However, by Lemma \ref{l-ineq}, we have 
  \[
  (K_{F}+\lambda B''+E'')|_{E''} \geq K_{E''}+\lambda {\rm Diff}_{E''}(B'')\equiv
  0\equiv (K_{F}+(1-\eta ) B''+E'')|_{E''} 
  \] 
  and so $1-\delta < 1-\eta \leq \lambda\leq 2/\beta <1$. But then 
  \[
  0<1-\frac 2 \beta < \delta<\frac 2 {\sqrt {V}}< \frac 2 {\sqrt {(-K_X)^2}}
  \] 
  which implies that $(-K_X)^2$ is bounded from above.
\end{proof}

\begin{lemma}\label{l-b} 
  Fix $\mathscr C\subset [0,1]$ a DCC set, then there exists an $\epsilon >0$ such
  that if $(X,B)$ is a projective klt surface such that $\rho (X)=1$, ${\rm coeff}(B)
  \subset \CC$ and $K_X+B\sim _\Q 0$, then $(X,B)$ is $\epsilon$ Kawamata log terminal.
\end{lemma}

\begin{proof} 
  Suppose that the claim is false. Then there is a sequence of pairs $(X_n,B_n)$ as
  above with total discrepancy $a(X_n,B_n)=\epsilon _n-1$ such that $\epsilon _n$ is
  a decreasing sequence with limit $0$. Let $\CC '=\CC \cup \{ 1-\epsilon _n \}_{n\in
    \N}$ then $\CC '$ is a DCC set.  Suppose that $(X_n,B_n)$ does not contain a
  component of coefficient $1-\epsilon_n$. Let $\nu :X'\to X=X_n$ be a projective
  birational morphism extracting the corresponding divisor $E$ so that $\rho
  (X'/X)=1$ and the exceptional divisor is $E$. We may write $K_{X'}+B'+eE=\nu
  ^*(K_X+B)$ where $e=1-\epsilon_n$.  Since $\rho (X')=2$ there is a second extremal
  ray $R_2$ (here $R_1=[E]$). Since $(K_{X'}+B')\cdot R_2=-eE\cdot R_2<0$, it follows
  that $R_2$ is $K_{X'}+B'$ negative and hence it can be contracted. Let $\mu: X'\to
  X''$ be the corresponding contraction.  If $\dim X''=1$, then let $F\simeq \mathbb P
  ^1$ be a general fiber. We have \[0=\deg (K_{X'}+B'+eE)|_F=-2+(B'+eE)\cdot F=-2
  +\sum b_i +e\] Since $b_i\in \mathscr C$, it is easy to see that $e=1-\epsilon _n$ is constant
  for $n\gg 0$ which is impossible.
  Therefore, we may assume that $X'\to X''$ is a birational contraction.  Then
  $K_{X''}+B''+eS=\mu_*(K_{X'}+B'+eE)\sim _\Q 0$. Replacing $X$ by $X''$ and $B$ by
  $B''+eS$, we may assume that $B$ contains a component $S$ of coefficient
  $e=1-\epsilon_n$.

  Write $B=B'+eS$. We then have \[\epsilon S^2=(1-e)S^2=(K_X+B'+S)\cdot S={\rm
    deg}(K_S+{\rm Diff }_S(B'))\geq \beta -2\] where $\beta={\rm min}\{ \sum
  b_i'|b_i'\in D(\CC ' ),\ \sum b_i'>2\}>2$. But then \[(-K_X)^2=B^2\geq
  (1-\epsilon)^2S^2\geq \frac {(1-\epsilon)^2}{\epsilon}(\beta -2)\] where $\lim
  _{n\to \infty}(1-\epsilon_n)^2(\beta -2)/\epsilon_n =+\infty $ contradicting Lemma
  \ref{l-a}.
\end{proof}

\begin{lemma}\label{l-c} 
  Fix $\mathscr C\subset [0,1]$ a DCC set, then there exists a constant $\delta >0$
  such that if $(X,B=\sum _{i=1}^rb_iB_i )$ is a klt surface such that $K_X+B$ is
  big, and $b_i \in \CC$, then $K_X+(1-\delta )B$ is big.
\end{lemma}
\begin{proof} If this were not the case, then there is a sequence of klt surfaces
  $(X_n,B_n)$ and a decreasing sequence of numbers $\delta _n>0$ such that $\lim
  \delta _n=0$ and $\kappa (K_{X_n}+(1-\delta _n)B _n)\in \{0,1\}$. After running a
  ($K_{X_n}+(1-\delta _n)B _n$)-minimal model program, we may assume that
  $K_{X_n}+(1-\delta _n)B _n$ is nef.  Now we run a $K_{X_n}$-minimal model
  program. After finitely many divisorial contractions, we may assume that we have a
  Mori fiber space $f:X'_n\to Z_n$. Since each divisorial contraction is
  automatically $K_{X_n}+(1-\delta _n)B _n$-trivial (see \cite[5.1, 5.2]{HMX14}), we
  may assume that $K_{X'_n}+(1-\delta _n)B '_n$ is nef and $f$ is $K_{X'_n}+(1-\delta
  _n)B '_n$-trivial.

  If $\dim Z_n=1$, let $F_n\simeq \mathbb P ^1$ be a general fiber. We have
  \[
  0=(K_{X'_n}+(1-\delta _n)B'_n)\cdot F_n=-2+(1-\delta _n)\sum n_ib_i
  \] 
  where $b _i\in \CC$ and $n_i\in \N$. Note that $(1-\delta _n)B'_n\cdot F_n\ne
  0$. Therefore $2/(1-\delta _n)$ is a decreasing sequence contained in the DCC set
  $\{\sum n_ib_i|n_i\in \N,\ b_i\in \CC \}$. Thus $\delta _n$ is evenytually constant
  as required.

  If $\dim Z_n=0$, then $\rho (X'_n)=1$ and $-K_{X'_n}$ is ample.  Since
  $K_{X'_n}+(1-\delta _n)B '_n\equiv 0$ and the coefficients of $(1-\delta _n)B '_n$
  belong to a DCC set, say $\mathscr C'$, by Lemma \ref{l-b} there exists an
  $\epsilon >0$ such that each $(X'_n,(1-\delta _n)B '_n)$ is $\epsilon$-klt and so
  by Lemmas \ref{l-d} and \ref{l-rho}, there is an integer $N>0$ such that $NK_{X'_n}$
  is Cartier. Now consider
  \[
  N(-K_{X'_n})^2=-(1-\delta _n)B'_n\cdot NK_{X'_n}.
  \] 
  Since $NK_{X'_n}$ is Cartier (and $K_{X'_n}$ is a Weil divisor), by Lemma \ref{l-a} $N(-K_{X'_n})^2\in \{1,2,3,\ldots
  , NV\}$ a finite set of positive integers.  Therefore, after passing to a
  subsequence, we may assume that $N(-K_{X'_n})^2$ is constant. But then
  $N(-K_{X'_n})^2/(1-\delta _n)$ cannot be an integer for $n\gg 0$ and this is a
  contradiction since $B'_n\cdot (-NK_{X'_n})\in \Z$.
\end{proof}

\begin{remark} 
  An effective version of \eqref{l-c} is proven in \cite[4.6]{AM03}.
\end{remark} 

\begin{lemma}\label{l-rho} 
  Fix $\epsilon>0$ then there exists a constant $\varrho=\varrho(\epsilon)$ such that
  if $(X,B)$ is a projective $\epsilon$-log canonical surface and $-(K_X+B)$ is nef,
  then ${\rm rk}\ {\rm Pic}(X)\leq \varrho$. In particular the number of exceptional
  divisors of negative discrepancy $a_E(X,B)<0$ is at most
  $\varrho$. 
\end{lemma}

\begin{proof} 
  Let $f:X'\to X$ be a projective birational morphism such that
  $K_{X'}+B'=f^*(K_X+B)$ where $B'\geq 0$ and $a_E(X',B')\geq 0$ for any divisor $E$
  exceptional over $X'$ (in other words $f$ extracts precisely the divisors of
  negative discrepancy $a_E(X,B)<0$). Clearly $-(K_{X'}+B')$ is nef, $X'$ is smooth
  and ${\rm coeff}(B')\in (0,1-\epsilon]$. By \cite[Theorem 6.3]{Alexeev94} (see also
  \cite[Theorem 1.8]{AM03}) there exists a constant $\varrho=\varrho(\epsilon)$
  such that $\rho (X)\leq \rho (X')\leq \varrho$. Finally the number of exceptional
  divisors of negative discrepancy is just $\rho (X')-\rho (X)\leq \rho (X')-1$ and
  the lemma follows.
\end{proof}

\begin{lemma}\label{l-d} 
  Fix $k\in \N$ and $\epsilon >0$. There exists an integer $N=N(k,\epsilon)$ such
  that if $(X,B)$ is an $\epsilon$-klt surface singularity such that the number of
  exceptional divisors of discrepancy $a_E(X,B)<0$ is $\leq k$ then $NK_X$ is Cartier
  and $NG$ is Cartier for any integral Weil divisor $G$ contained in the support of
  $B$.
\end{lemma}

\begin{proof} (See also \cite{Alexeev94} and \cite{AM03}) Let $\nu :X'\to X$ be a
  partial resolution extracting all divisors of discrepancy $a_E(X,B)<0$, in
  particular $X'$ has at most du Val singularities which are not contained in the
  support of $B'$ where $K_{X'}+B'=\nu ^*(K_X+B)$. By the classification of klt
  singularities \cite{Alexeev81}, the weights of each curve in the corresponding
  graph are bounded by $2/\epsilon$ (cf.~\cite[Proof of 7.5]{Alexeev94}) and so there
  are only finitely many possibilities for the corresponding graph. Let $G=K_X$ or
  $G$ be a component of the support of $B$ and $G'$ its strict transform. Then we may
  write $\nu ^* G=G'+\sum e_iE_i$ where the $E_i$ are exceptional divisors and the
  denominators of the $e_i$ divide $t=|{\rm det}(E_k\cdot E_{k'})|$. But then $t(
  G'+\sum e_iE_i)$ is integral.  Since $X'$ has only du Val singularities, $t\nu ^*G$
  is Cartier.  By the Basepoint-free theorem $tG$ is Cartier (cf. the proof of
  \cite[4.7]{AM03}).
\end{proof}

\end{subsection}

\begin{subsection}{Proof of Theorem~\ref{t-dcc}}
  We follow some ideas from \cite{AM03} and \cite{Alexeev94} applying techniques from
  \cite{BHMM11}. We may assume that $\mathscr C \supset \{ 1-\frac 1 m|m\in \N\}\cup \{1\}$.
  Note that it suffices to prove the theorem for log canonical pairs. To see this,
  consider an slc model $(X,B)$ and its normalization $\nu :\cup X_i\to X$. Writing
  $K_{X_i}+B_i=(\nu |_{X_i})^*(K_X+B)$, we have log canonical models $(X_i,B_i)$ such
  that ${\rm coeff}(B_i)\in \mathscr C$ and $(K_X+B)^2=\sum (K_{X_i}+B_i)^2$. The
  claim now follows easily since if $\mathscr D$ is a DCC set, then so is $\mathscr D
  '=\{ \sum d_i|d_i\in \mathscr D\}$.

  Suppose now, by way of contradiction, that $(X_m,B_m)$ is a sequence of slc
  surfaces defined over the algebraically closed field $L_m$ of characteristic
  $p_m>0$, such that $\coeff(B_m)\subseteq \mathscr C$ and
  \begin{equation}\label{eq-a}
    {\rm vol}(K_{X_m}+B_m)>{\rm  vol}(K_{X_{m+1}}+B_{m+1}).
  \end{equation} 
  In particular we may fix a constant $V>0$ such that $ {\rm vol}(K_{X_m}+B_m)\leq V$
  for all $m\in \N$.

  Passing to a log resolution, we may assume that $(X_m,B_m)$ is an snc pair. In
  fact, given a birational morphism $X'_m\to X_m$ let $B'_m$ be the strict transform
  of $B_m$ plus the exceptional divisor 
  so that ${\rm vol }(K_{X'_m}+B'_m)={\rm vol}(K_{X_m}+B_m)$ (cf.\ Proposition
  \ref{p-vols}) and $\coeff(B'_m)\subseteq \CC$. Then we replace $(X_m,B_m)$ by
  $(X'_m,B'_m)$.  Since $(X_m,B_m)$ is a snc pair and bigness is an open condition,
  replacing the coefficients that equal 1 by $1-\frac 1 r$ for some $r\gg 0$, we may
  assume that $(X_m,B_m)$ is klt.

  \begin{claim}\cite[7.6]{Alexeev94} 
    We may assume that the pairs $(X_m,B_m)$ are log birationally bounded, i.e.,
    there exists a constant $d>0$ and birational maps $f_m:X_m \dasharrow Z_m$ and
    very ample divisors $H_m$ on $Z_m$ such that $H_m^2\leq d$ and $H_m\cdot
    B_{Z_m}\leq d$ where $B_{Z_m}$ is the sum of the strict transform of $B_m$ and
    the $Z_m\dasharrow X_m$ exceptional divisors.
  \end{claim}

  By Lemma \ref{l-c} it follows easily that there is a finite set of rational numbers
  $\mathscr C'$ depending only on $\mathscr C$ and divisors $0\leq D_m\leq B_m$ such
  that $K_{X_m}+D_m$ is big and $\coeff(D_m)\subseteq \mathscr C '$. By
  \cite[7.3]{Alexeev94} we may also assume that the number of components of $B_m$ is
  bounded by a constant (depending only on $\mathscr C$). Let $\mu _m:X_m\to X'_m$ be
  a minimal model for $K_{X_m}+D_m$ and $D'_m=(\mu _m)_*D_m\leq B'_m=(\mu
  _m)_*B_m$. Note that $K_{X'_m}+D'_m$ is klt and big.  Let $f'_m:X'_m\to Z_m$ be the
  be the corresponding log canonical model for $K_{X_m}+D_m$ and $f_m:X_m\to Z_m$ the
  induced morphism. Since the number of components of $D_m$ is bounded, it follows
  easily that the number of divisors $E$ over $Z_m$ of discrepancy
  $a_E(Z_m,(f_m)_*D_m)<0$ is bounded from above.  By Lemma \ref{l-d}, there exists an
  integer $N>0$ depending only on $\CC$ and $V$ such that $G_m=N(K_{Z_m}+(f_m)_*D_m)$
  is ample and Cartier. By Corollary \ref{c-b}, $|K_{X_m}+qf_m^*G_m|$ is birational for
  all $q\geq 18$ and hence so is $|K_{X_m}+18N(K_{X_m}+D_m)|$ .  Since $B_m\geq D_m$,
  it follows that $|K_{X_m}+18N(K_{X_m}+B_m)|$ is birational. Since ${\rm
    vol}(18N(K_{X_m}+B_m))\leq (18N)^2V$, by Theorem \ref{t-bb} (with $q=18N$) it follows
  that the pairs $(X_m,B_m)$ are log birationally degree bounded.

  \begin{claim} 
    We may assume that $f_m:X_m\to Z_m$ is a morphism given by a finite sequence of
    blow ups along smooth strata of $(Z_m,\wh B_{Z_m})$ where $\wh
    B_{Z_m}=(B_{Z_m})_{\rm red}$ and that $(Z_m, \wh B_{Z_m})$ is an \sncp and is
    degree bounded.
  \end{claim}

  \begin{proof}
    Let $([Z_m],[\wh B_{Z_m}])$ be the internal pair associated to the sequence of
    pairs $(Z_m,\wh B_{Z_m})$ where $\wh B_{Z_m}=(B_{Z_m})_{\rm red}$.  Since $Z_m$
    and $\wh B_{Z_m}$ are degree bounded, it follows by Proposition \ref{p-ultra}
    that there exists a pair $(Z, \wh B)$ defined over $k=[L_m]$ such that $(Z, \wh
    B)_{\rm int}=([Z_m],[\wh B_{Z_m}])$.  Let $\nu :Z'\to Z$ be a log resolution of
    $(Z,\wh B)$ and $\wh B'=\nu ^{-1}_*B+{\rm Ex}(\nu)$. If $(Z',\wh B')_{\rm
      int}=([Z_m'],[\wh B'_{Z'_m}])$, then it is easy to see that $Z'_m$ and $\wh
    B'_{Z'_m}$ are degree bounded (for almost all $m\in \N$).  Replacing $(X_m,B_m)$
    by an appropriate birational model $(X'_m,B'_m)$, we may assume that
    $f'_m:X'_m\to Z'_m$ is a morphism with $(f'_{m})_*(B'_m)\leq \wh B'_{Z'_m}$.
    Replacing $(X_m,B_m)$ by $(X'_m,B'_m)$ and $X_m\to Z_m$ by $X'_m\to Z'_m$ we may
    assume that $f_m$ is a morphism, $(Z_m,B_{Z_m})$ is an \sncp and is degree
    bounded.

    Let $(Z,\wh B)$ be the projective pair (over $k$) defined above so that $(Z, \wh
    B)_{\rm int}=([Z_m],[\wh B_{Z_m}])$ where as above $\wh B$ and $\wh B_{Z_m}$
    denote the reduced divisors.  Let  $X''_m\to Z_m$ be a finite sequence of strata such that every divisor $E$ on $X_m$ 
    of discrepancy $a_E(Z_m,B_{Z_m})<0$ is a divisor on $X''_m$. Let $B''_m$ be the strict transform of $B_m$ plus the sum of all $X''_m\to Z_m$ exceptional divisors which are not also $X''_m\dasharrow X_m$ exceptional, taken with coefficient $1-1/r$ for some $r\gg 0$.  Then one can see
    easily that ${\rm vol}(K_{X_m}+B_m)={\rm vol}(K_{X''_m}+B''_m)$ (cf.
    Proposition \ref{p-vols}).  Replacing $(X_m,B_m)$ by the
    pair $(X''_m,B''_m)$, we
    may therefore assume that each $X_m$ is obtained from $Z_m$ via a
    finite sequence of blow ups along smooth strata of $(Z_m,B_{Z_m})$.
  \end{proof}

  For almost all $m\in \N$, the strata of $(Z_m,\wh B_{Z_m})$ are in one-to-one
  correspondence with the strata of $(Z,\wh B)$ (cf.\ \cite[3.8]{BHMM11}).
  Therefore, we define $(X^m,B^m)$ by blowing up the corresponding strata on $(Z,\wh
  B)$ and choosing the coefficients of $B^m$ to match those of $B_m$.  Let $\nu$ be
  any divisorial valuation over $Z$.  Since the coefficients belong to a DCC set,
  after passing to a subsequence, we may assume that the sequence $ \mathbf
  M_{B^m}(\nu)$ is non decreasing and hence that $\lim \mathbf M_{B^m}(\nu)$ exists.
  Notice that if $ \mathbf M_{B^m}(\nu)\ne 0$, then $\nu$ corresponds to either a
  component of $\wh B$ or to a divisor exceptional over $Z$. If moreover $ \mathbf
  M_{B^m}(\nu)\not\in \{ 0, 1\}$ then the corresponding divisor is obtained by
  blowing up $Z$ along some strata of $\wh B$. Therefore, there are only countably
  many divisorial valuations $\nu$ for which $\mathbf M_{B^m}(\nu)\ne \mathbf
  M_{B^k}(\nu)$.  By a standard diagonalization argument, we may assume that after
  passing to a subsequence, there is a well defined $b$-divisor over $Z$ defined by
  $\mathbf B(\nu )=\lim \mathbf M_{B^m}(\nu)$ for any valuation $\nu$ over $Z$. Let
  $\Phi :=\mathbf B _Z$. Let $B^m_Z$ be the pushforward of $B^m$ to $Z$, then
  $\mathbf B _Z=\lim B^m_Z$ Since $\mathscr C$ satisfies the DCC, it follows that
  \begin{equation}\label{eq-0} B^m_Z\leq \Phi\qquad {\rm for\ almost\ all \ }m\in
    \N.\end{equation}
  \begin{claim}\label{c-1} We may assume that \begin{equation}\label{e-cut}
      \mathbf L_{\Phi }\leq \mathbf B,\end{equation}
    where $\Phi =\mathbf B _Z$.
  \end{claim}
\begin{proof}
  We follow the proof of \cite[5.7]{HMX13} checking that our choices do not affect
  the volume of $K_{X_m}+B_m$.  Let $(Z',\mathbf B')$ be the reduction of $(Z,\mathbf
  B)$ defined in \cite[5.7]{HMX13}, so that if $\Phi'=\mathbf B'_{Z'}$, then we have
  the inequality of b-divisors
  \begin{equation*}
    \mathbf L_{\Phi '}\leq \mathbf B'.
  \end{equation*}
  Recall that the reduction $(Z',\mathbf B')$ is given by a finite sequence of cuts
  where a cut is defined as follows: given a birational morphism of smooth projective
  varieties $\mu :Z'\to Z$ and a subset $\Sigma$ of the $\mu $ exceptional divisors,
  for every valuation $\sigma \in \Sigma$, let $\Gamma _\sigma = (\mathbf L_\Phi
  \wedge \mathbf B )_{Y_\sigma}$, where $Y_\sigma \to Z$ is the divisorial
  contraction of the divisor over $Z$ corresponding to $\sigma$ which defined in
  \cite[5.4]{HMX13} and $\Phi =\mathbf B _Z$. Let $\Theta = \wedge _{\sigma \in
    \Sigma}(\mathbf L_{\Gamma _\sigma})_{Z'}$, the minimum of the divisors $(\mathbf L_{\Gamma
  _\sigma})_{Z'}$.  The cut of $(Z, \mathbf B)$, associated to $Z' \to Z$ and
  $\Sigma$, is the pair $(Z', \mathbf B')$, where $\mathbf B' = \mathbf B\wedge
  \mathbf M_\Theta$, so that $\mathbf B' _{Z'}=\Theta \wedge \mathbf B_{Z'}$ and
  $\mathbf B'(\nu)=\mathbf B(\nu )$ for any valuation $\nu$ corresponding to an
  exceptional divisor over $Z'$. We may assume that $Z'\to Z$ is given by a finite
  sequence of blow ups along strata of $\wh B$ and so we let $Z'_m\to Z_m$ be
  obtained by the corresponding sequence of blow ups along strata of $B_{Z_m}$ for
  almost all $m\in \N$.
  After possibly blowing up $X^m$ and replacing $B^m$ by its strict transform plus
  the exceptional divisor, we may assume that $X^m\to Z$ factors via a morphism
  $X^m\to Z'$ and similarly we have morphisms $X_m\to Z'_m$ for almost all $m\in
  \mathbb N$.

  Now consider the divisors ${B'}^m$ on $X^m$ defined by ${B'}^m=B^m\wedge (\mathbf M
  _{\Theta ^m})_{X^m}$ where $\Theta ^m =\wedge _{\sigma \in \Sigma} (\mathbf L
  _{\Gamma ^m_\sigma})_{Z'}$, $\Gamma _\sigma ^m=(\mathbf L_{B^m_Z})_{Y_\sigma
  }\wedge B _{Y_\sigma }$ where $B^m_Z$ and $ B ^m_{Y_\sigma }$ are the pushforwards
  of $B^m$ to $Z$ and $Y_\sigma$. Then, as in the proof of \cite[5.7]{HMX13}, we may
  assume that $\mathbf B'=\lim \mathbf M_{{B'}^m}$.  Let $B'_m$ be the divisors on
  $X_m$ corresponding to ${B'}^m$. We will show that ${\rm vol}(K_{X_m}+B_m)={\rm
    vol}(K_{X_m}+B'_m)$.  Assuming this, we may replace $B_m$ by $B'_m$ and the claim
  follows.

  We define $\Phi _m=B_{Z_m}$ and \[\Gamma _{m,\sigma}=(\mathbf L_{\Phi
    _m})_{Y_{m,\sigma}}\wedge B_{Y_{m,\sigma}},\] where $B_{Y_{m,\sigma}}$ is the
  pushforward of $B_m$ to $Y_{m,\sigma}$.  Let $\Theta _m:=\wedge _{\sigma \in
    \Sigma}(\mathbf L_{\Gamma _{m,\sigma}})_{ Z'_m}$. It is easy to see that the
  divisors $B_{Z_m}$, $B_{Y_{m,\sigma}}$, $\Gamma _{m,\sigma}$ and $\Theta _m$
  correspond to the divisors $B^m_Z$, $B^m_{Y_\sigma}$, $\Gamma ^m_\sigma$ and
  $\Theta ^m$ so that we have \[B'_m=B_m\wedge (\mathbf M_{\Theta _m})_{X_m}.\] It
  then follows that
  \begin{equation}\label{e-vol}
    B_m\wedge (\mathbf L_{\Theta _m})_{X_m}\leq
    B'_m=B_m\wedge (\mathbf M_{\Theta _m})_{X_m}\leq B_m.
  \end{equation} 
  Thus, by \eqref{e-vol} and Proposition \ref{p-vols}
  \[
  {\rm vol}(K_{X_m}+B_m)={\rm vol}(K_{X_m}+B_m').
  \]

  Replacing $B_m$ by $B'_m$ the claim  follows.
  \end{proof}
  \begin{claim}\label{c-vol} 
    For almost all $m\in \N$ we have ${\rm vol}(K_{Z_m}+t\Phi _{Z_m})={\rm
      vol}(K_{Z}+t\Phi )$ for all $t\in [0,1]\cap \mathbb Q$.
  \end{claim}
  \begin{proof}
    There are finitely many birational morphisms $\{\psi ^i:Z\to W^i\}_{i\in I}$ such
    that for any $t\in [0,1]$, there exists an $i\in I$ such that $\psi ^i$ is a
    minimal model for $K_{Z}+t\Phi $. Let $[\psi ^i_m]:[Z_m]\to [W^i_m]$ be the
    corresponding morphism of internal schemes. It is easy to see that for almost all
    $m\in \N$ this is a minimal model for $K_{Z_m}+t\Phi _{Z_m}$ and
    $(K_{W^i_m}+t\psi ^i_{m,*} \Phi _{Z_m})^2=(K_{W^i}+t\psi _* \Phi_{Z})^2$ (cf.\
    Lemma \ref{l-um}). Therefore, the claim follows.
  \end{proof}
  
  Now we observe that
  \[
  {\rm vol}(K_{X_m}+B_m)\leq {\rm vol}(K_{Z_m}+B_{Z_m}) \leq {\rm vol}(K_{Z_m}+\Phi
  _{Z_m}) ,
  \] 
  where the first inequality follows as $B_{Z_m}$ is the pushforward of $B_m$, and
  the second as $B_{Z_m}\leq \Phi _{Z_m}$ cf.~\eqref{eq-0}.

  On the other hand, for any $\epsilon >0$, the pair $(Z, (1-\epsilon )\Phi )$ is klt
  with simple normal crossings and hence there is a terminalization $h:Y\to Z$ (given
  by a finite sequence of blow ups along strata of $(Z, (1-\epsilon )\Phi )$) so that
  $(Y, \Psi:=\mathbf L_{(1-\epsilon)\Phi,Y} )$ is terminal. We have that for some
  $\lambda>0$,
  \begin{equation}\label{e-leq}
    \Psi \leq (1-\lambda)\mathbf L_{\Phi,Y}\leq \mathbf L_{\Phi,Y}\leq \mathbf B_Y
  \end{equation}
  where the last inequality follows from \eqref{e-cut}.  For almost all $m\in \N$ we
  may consider $h_m:Y_m\to Z_m$ given by the same sequence of blow ups along strata
  of $(Z_m,\Phi _{Z_m})$. Then, denoting by $\mathbf B _{Y_m}$ and $\Psi _m$ the
  divisors on $Y_m$ corresponding to $\mathbf B _Y$ and $\Psi$, since $\mathbf B_Y=
  \lim \mathbf M _{B^m,Y}$, comparing coefficients of divisors on $Y$, by
  \eqref{e-leq}, for infinitely many $m\in \N$ we have $\Psi \leq \mathbf M_{B^m,Y}$
  and hence also $\Psi _m\leq \mathbf M _{B_m,Y_m}.$ It follows that then
  \begin{multline}\label{e-long}    
    \begin{aligned}
      {\rm vol}( K_{Z}+(1-\epsilon)\Phi ) ={\rm vol}(
      K_{Z_m}+(1-\epsilon)\Phi_{Z_m})={\rm vol}( K_{Y_m}+\Psi _m) \leq  \hfill \\
      \leq {\rm vol}( K_{Y_m}+\mathbf M _{B_m,Y_m})= {\rm vol}(K_{X_m}+B_m),\hskip-6em
    \end{aligned} 
  \end{multline}
  where the first (in)equality follows from Claim \ref{c-vol}, the second since
  $(Y_m,\Psi _m)$ is a terminalization of $(Z_m, (1-\epsilon)\Phi_{Z_m})$ (observe
  that $\Psi _m=(\mathbf L_{(1-\epsilon)\Phi_{Z_m}})_{Z_m}$ and apply Proposition
  \ref{p-vols}), the third since $\Psi _m\leq \mathbf M _{B_m,Y_m}$, and the fourth
  by Proposition \ref{p-vols}.  Taking the limit as $\epsilon \to 0$, by \eqref{eq-a}
  we obtain
  \[
  {\rm vol}( K_{Z}+\Phi)\leq \lim {\rm vol}(K_{X_m}+B_m)< {\rm
    vol}(K_{X_m}+B_m).
  \] 
  Combining this with the above equations and Claim \ref{c-vol}, we
  have that 
  \begin{equation}\label{e-eq}
    {\rm vol}( K_{Z}+\Phi)< {\rm
      vol}(K_{X_m}+B_m) \leq {\rm vol}(K_{Z_m}+\Phi _{Z_m})={\rm
      vol}(K_{Z}+\Phi)
  \end{equation} 
  for infinitely many $m$. This is the required contradiction and it completes the
  proof of Theorem~\ref{t-dcc}. \qed
\end{subsection}

\begin{subsection}{Proof of Theorem~\ref{t-m}} 
  It suffices to show that for any sequence of projective log canonical surfaces
  $(X_m,B_m)$ with fixed volume $(K_{X_m}+B_m)^2=v$ and ${\rm coeff}(B_m)\subseteq
  \mathscr C$, there exists an integer $r>0$ such that $r(K_{\bar X_m}+\bar B_m)$ is
  very ample where $\phi _m:X_m\to \bar X _m$ is the log canonical model of
  $(X_m,B_m)$ and $\bar B _m=\phi _{m,*}B_m$. Arguing by contradiction (and passing
  to a subsequence), assume that $m!(K_{\bar X_m}+\bar B_m)$ is not very ample for
  all $m>0$.  Following the notation in the above proof, let $Z\to W$ be the log
  canonical model of $(Z,\Phi=\mathbf B _{Z})$. Let $[Z_m]\to [W_m]$ be the
  corresponding morphism of internal schemes, so that for almost all $m\in \N$ we
  have morphisms $h_m:Z_m\to W_m$ which are log canonical models for $({Z_m},\Phi
  _{Z_m})$. We have \[{\rm vol}(K_{X_m}+B_m)={\rm vol}(K_{Z_m}+\Phi _{Z_m})\geq {\rm
    vol}(K_{Z_m}+B_{Z_m})\geq {\rm vol}(K_{X_m}+B_m),\] where the first (in)equality
  follows since all inequalities in \eqref{e-eq} are actually equalities, the second
  since $\Phi _{Z_m}\geq B_{Z_m}$ and the last as $K_{Z_m}+B_{Z_m}$ is the
  pushforward of $K_{X_m}+B_m$.

  Since $K_{Z_m}+B_{Z_m}$ is big and has a log canonical model, and $\Phi _{Z_m}\geq
  B_{Z_m}$, it follows by \cite[2.2.2]{HMX15} that $Z_m\to W_m$ is a log canonical
  model for $(Z_m,B_{Z_m})$.  In particular, $(h_m)_*B_{Z_m}=(h_{m})_*\Phi _{Z_m}$ is
  rational and hence so is $h_* \Phi $. But then by the result over the fixed field
  $k$ (see \cite[9.2]{Alexeev94}), we know that there is an integer $r$ depending only on $\const={\rm
    vol}(K_{X_m}+B_m)$ and $\mathscr C$ such that $r(K_W+h_*\Phi )$ is Cartier and
  very ample. But then $r(K_{W_m}+h_{m,*}\Phi )$ is Cartier and very ample for
  infinitely many $m>0$. This is the required contradiction and the assertion of
  Theorem~\ref{t-m} follows. \qed
\end{subsection}

\begin{subsection}{Proof of Corollary~\ref{c-m}}
  We may assume that $1-\epsilon \in \mathscr C$.  It suffices to show that any
  sequence of $\epsilon$-log canonical projective pairs $(X_m,B_m)$ with $\dim
  X_m=2$, ${\rm coeff}(B_m)\in \mathscr C$, $K_{X_m}+B_m$ nef and big and ${\rm
    vol}(K_{X_m}+B_m)\leq v$ is degree bounded.

  Following the proof of Theorem~\ref{t-dcc}, we may assume that $(Z,\mathbf B_Z)$ is
  an \sncp with coefficients $\leq 1-\epsilon$. Replacing $Z$ by an appropriate
  birational model, we may in fact assume that $(Z,\mathbf B_Z)$ is terminal and
  hence so are $(Z_m, B_{Z_m})$. But then ${\rm vol}(K_{X_m}+B_m)=\vol
  (K_{Z_m}+B_{Z_m})$ for almost all $m\in \N$ by Proposition \ref{p-vols} and so we
  may assume that $(X_m,B_m)=({Z_m},B_{Z_m})$. Notice that we have replaced $X_m$ by
  an appropriate birational model and $B_m$ by its strict transform plus the
  exceptional divisors with coefficient $(1-\epsilon)$, hence $K_{X_m}+B_m$ may no
  longer be nef. Let $B^m$ be the divisors on $Z$ corresponding to $B_m$ on
  $X_m$. Since the support of $B^m$ has finitely many components and $\CC$ is a DCC
  set, after passing to a subsequence, we may assume that $B^m\leq B^{m+1}\leq
  B^{m+2}\leq \ldots \lim B^i=B^\infty$.  Let $B^\infty _m$ be the corresponding
  divisors on $X_m$, so that $B^\infty _m\geq B_m $. We claim that $K_Z+B^\infty$ is
  big. If this were not the case, then $Z$ would be covered by curves $C$ with
  $(K_Z+B^\infty )\cdot C\leq 0$. But then, the same would be true for $(X_m,B_m)$ as
  for almost all $m\in \N$ we have\[0\geq (K_Z+B^\infty )\cdot C=(K_{X_m}+B ^\infty
  _m)\cdot C_m\geq (K_{X_m}+B _m)\cdot C_m,\] where $(C)_{\rm int}=[C_m]$. This
  contradicts the fact that $K_{X_m}+B _m$ is big.  Since being big is an open
  condition, it follows that $K_Z+(1-\delta )B^\infty $ is big for all $0<\delta \ll
  1$ and we may assume that $(1-\delta )B^\infty\leq B^m\leq B^\infty$.  The set of
  all minimal/canonical models $Z\to W$ for pairs $(Z,G)$ with $ (1-\delta
  )B^\infty\leq G\leq B^\infty$ is bounded.  Arguing as in the proof of Claim
  \ref{c-vol} the corresponding rational maps $Z_m\to W_m$ give minimal/canonical
  models for $(X_m,B_m)$ for almost all $m\in \N$. Corollary~\ref{c-m} follows
  easily.
\end{subsection}

\begin{subsection}{Proof of Theorems~\ref{t-ssr} and \ref{t-ssr-three}}

  In the sequel we will use the following notation: If $f_m:X_m\to S_m$ is a morphism
  of schemes and $s\in S_m$ a point, then $X_{m,s}$ denotes the fiber
  $(X_m)_s=f_m^{-1}(s)$. More generally, if $S'_m\to S_m$ is a morphism and $s\in
  S'_m$ a point, then $X_{m,s}$ denotes the fiber product $X_m\times
  _{S_m}\{s\}$.

\begin{theorem}\label{t-ssr-two}
  Fix a constant $\const\in\N$, a DCC set $\CC \subset [0,1]\cap\Q$. For each $m\in
  \Z$, let $L_m$ be an algebraically closed field of $\chara L_m= p_m>0$ such that
  $\lim p_m=\infty$. Let $k=[[L_m]]$, $S$ a smooth $1$-dimensional scheme defined
  over $k$, and $S_{\rm int}=[S_m]$ the corresponding internal scheme. Further let
  $(X_m,B_m)$ be a pair defined over $L_m$ such that $\dim X_m=3$, and let
  $f_m:X_m\to S_m$ be a projective morphism with connected fibers. Assume that for
  each $m\in\Z$, $\coeff(B_{m,\eta })\subseteq \CC$, $(X_{m,\eta }, B_{m,\eta })$ is
  semi-log canonical, and $K_{X_{m,\eta }}+ B_{m,\eta }$ is ample with $(K_{X_{m,\eta
    }}+ B_{m,\eta })^2=\const$ where $\eta$ denotes the generic point of $S_m$.

  Then there exist a finite separable morphism $S'\to S$, and a projective semistable
  family of semi-log canonical models $(X',B')\to S'$ such that considering the
  corresponding internal objects, 
  for an infinite subset $V\subseteq \Z$ and for each $m\in V$ there exist an
  induced 
  separable finite morphism $\sigma_m:S'_m\to S_m$, a projective morphism $X'_m\to
  S'_m$, and a pair $(X'_m,B'_m)$ such that $K_{X'_{m,s}}+B'_{m,s}$ is ample,
  $(X'_{m,s},B'_{m,s})$ 
  is semi-log canonical (in particular reduced) for all $s\in S'_m$, and $(X'_{m,\eta
    '},B'_{m,\eta '})$ is isomorphic to $(X_{\eta '}, B_{\eta'})$ for where $\eta \in
  S_m$ and $\eta '\in S'_m$ are the generic points.
\end{theorem}

\begin{proof}
  Let $F_m=K(S_m)$ be the field of rational functions of $S_m$, $\ol F_m$ its
  algebraic closure, and $(\ol X_{m,\eta},\ol B_{m,\eta})$ the geometric general
  fiber of $f_m$ (obtained by the base change $\Spec \ol F_m\to S_m$).  Since
  $\coeff(\ol B_{m,\eta})\subseteq \CC$ and $\vol (K_{\ol X_{m,\eta}}+\ol
  B_{m,\eta})=\vol (K_{X_{m,\eta}}+B_{m,\eta})=\const$, by Theorem~\ref{t-m} there is
  a fixed integer $r>0$ (independent of $m$) such that $r(K_{\ol X_{m,\eta}}+\ol
  B_{m,\eta})$ is very ample.

  As $X_{m,\eta}$ and the components of $B_{m,\eta}$ have bounded degree (cf.\
  Theorem~\ref{t-bb}), there exists, by Proposition \ref{p-ultra}, a pair
  $(X^\circ,B^\circ)$ defined over $F=[F_m]$ such that $(X^\circ,B^\circ)_{\rm
    int}=([X_{m,\eta}],[(B_{m,\eta})_{\rm red}])$.  Since $r(K_{\ol X_{m,\eta}}+\ol
  B_{m,\eta})$ is very ample and hence in particular Cartier, it follows that
  $\coeff(B_m)\subseteq \left\{\left.\dfrac ar \right\vert 0\leq a\leq r\right\}$ and
  so, after passing to an infinite subset of $\Z$, we may assume that actually
  $B^\circ_{\rm int}=[B_{m,\eta}]$.

  Let $k=[L_m]$ and note that it is algebraically closed of $\chara k=0$ by
  \cite[2.4.1, 2.4.2]{Schoutens10}.  By definition $S_{\rm int}=[S_m]$, so
  $K(S)= F=[F_m]$ by Theorem~\ref{l-ultra} and the construction of the
  functor $Z\mapsto Z_{\rm int}$. 
 Since $\chara k=0$, after a possible
  base change, resolving and taking the relative semi-log canonical model, one
  obtains a semistable family of semi-log canonical models $(X',B')\to S'$ with a
  finite separable morphism $S'\to S$.  Considering the corresponding internal
  objects $X'_{\rm int}=[X'_m]$, $B'_{\rm int}=[B'_m]$, and $S'_{\rm int}=[S'_m]$
  proves Theorem~\ref{t-ssr-two}.
\end{proof}

\begin{corollary}[\protect{=Theorem~\ref{t-ssr}}]
  \label{cor:stable-red-one}
  Fix a constant $\const\in\N$ and a DCC set $\CC \subset [0,1]\cap\Q$. Then there
  exists a number $p_0>0$ such that if $L$ is an algebraically closed field of
  characteristic $p>p_0$, $(X,B)$ a pair defined over $L$ such that $\dim X=3$,
  $f:X\to S=\Spec L[[t]]$ a projective morphism with connected fibers
  such that, $\coeff(B_{m,\eta })\subseteq \CC$, $(X_{m,\eta }, B_{m,\eta })$ is
  semi-log canonical, and $K_{X_{m,\eta }}+ B_{m,\eta }$ is ample with $(K_{X_{m,\eta
    }}+ B_{m,\eta })^2=\const$, then there exist a 
  separable finite morphism $S'\to S$, a projective morphism $f':X'\to S'$, and a
  pair $(X',B')$ such that considering the corresponding internal objects for an
  infinite subset $V\subseteq \Z$ and for each $m\in V$,
  $(X'_{m,s},B'_{m,s})$ is semi-log canonical and $K_{X'_{m,s}}+B'_{m,s}$ is ample
  for all $s\in S'$, and $(X'_{m,\eta '},B'_{m,\eta '})$ is isomorphic to $(X_{\eta
    '},B_{\eta '})$ where $\eta \in S_m$ and $\eta '\in S'_m$ are the generic points.
\end{corollary}

\begin{corollary}\label{cor:stable-red-two}
  Fix constants $\const, g\in\N$ and a DCC set $\CC \subset [0,1]\cap\Q$. Then there
  exists a number $p_0>0$ such that if $L$ is an algebraically closed field of
  characteristic $p>p_0$, $(X,B)$ a pair defined over $L$ such that $\dim X=3$,
  $f:X\to S$ a projective morphism with connected fibers, where $S$ is a smooth curve
  over $L$ whose geometric genus is at most $g$,
  such that,  $\coeff(B_{m,\eta })\subseteq \CC$, $(X_{m,\eta },
  B_{m,\eta })$ is log canonical, and $K_{X_{m,\eta }}+ B_{m,\eta }$ is ample with $(K_{X_{m,\eta }}+
  B_{m,\eta})^2=\const$, then there exist a 
  separable finite morphism $S'\to S$, a projective morphism $f':X'\to S'$, and a
  pair $(X',B')$ such that  considering the corresponding internal objects,  for an infinite
  subset $V\subseteq \Z$ and for each $m\in V$,  $(X'_{m,s},B'_{m,s})$ is semi-log canonical and
  $K_{X'_{m,s}}+B'_{m,s}$ is ample for all $s\in S'$, and $(X'_{m,\eta '},B'_{m,\eta '})$ is
  isomorphic to $(X_{\eta '},B_{\eta '})$ where $\eta \in S_m$ and $\eta '\in S'_m$ are the generic points.
\end{corollary}

\begin{remark}
  Note that the situation of Corollary~\ref{cor:stable-red-two} arises for instance
  if $S$ is the ``reduction mod $p$'' of a fixed curve defined in characteristic zero.
\end{remark}

\begin{proof}[Proof of Corollaries~\ref{cor:stable-red-one} and
  \ref{cor:stable-red-two}]
  Let $L_m$ be a sequence of algebraically closed fields of characteristic $p_m>0$
  such that $\lim p_m=\infty$, $(X_m, B_m)$ a sequence of pairs and $X_m\to S_m$ a
  sequence of morphisms defined over $L_m$, where either $S_m=L_m[[t]]$ for each $m$
  or $S_m$ is a smooth curve over $L_m$ whose geometric genus is at most $g$ for each
  $m$.  Suppose that the conclusion of the appropriate corollary fails for each $m$.
  Let $k:=[L_m]$ and either let $S:=\Spec k[[t]]$ or let $S$ be the smooth curve
  provided by Proposition \ref{p-ultra} (cf.\ \cite[3.7]{BHMM11}). With these
  definitions the assumptions of Theorem~\ref{t-ssr-two} are satisfied and hence we
  obtain a contradiction. 
\end{proof}

\begin{theorem}[\protect{=Theorem~\ref{t-ssr-three}}]
  Fix a constant $\const\in \N$ and a DCC set $\mathscr C\subset [0,1]\cap \mathbb
  Q$. For each $m>0$ let $L_m$ be an algebraically closed field of characteristic
  $p_m>0$ such that $\lim p_m=\infty$ and let $k=[L_m]$.  Further let $(X_m,B_m)$ be
  a pair defined over $L_m$ such that $\dim X_m=3$, and let $f_m:X_m\to S_m$ be a
  projective morphism with connected fibers to a smooth curve. Assume that for each
  $m\in\Z$,  $\coeff(B_{m,\eta })\subseteq \CC$, $(X_{m,\eta }, B_{m,\eta })$ is
  semi-log canonical, and $K_{X_{m,\eta }}+ B_{m,\eta }$ is ample with $(K_{X_{m,\eta
    }}+ B_{m,\eta })^2=\const$ where $\eta$ is the generic point of $S_m$.

  Then 
  for all but finitely many $m$'s there exist a separable finite morphism $\sigma
  _m:S'_m\to S_m$, a projective morphism $X'_m\to S'_m$, and a pair $(X'_m,B'_m)$
  such that $(X'_{m,s},B'_{m,s})$ is semi-log canonical and $K_{X'_{m,s}}+B'_{m,s}$
  is ample for all $s\in S'_m$, and $(X'_{m,s},B'_{m,s})$ is isomorphic to
  $(X_{m,\sigma_m(s)},B_{m,\sigma_m(s)})$ for general $s\in S'_m$.
\end{theorem}

\begin{proof}
  Suppose that the conclusion of the theorem fails i.e. that it fails for infinitely
  many primes $p_m$. Passing to a subsequence, we may assume that the conclusion
  fails for every prime $p_m$ and we aim to find a contradiction.  Since the
  statement is local over the base, we may assume that $S_m={\rm Spec}(R_m)$ where
  $R_m$ is a DVR with closed point $s_m$.  Let $\widehat S_m$ be the formal
  neighborhood of $s_m\in S_m$ and $\widehat X_m$ be the formal neighborhood of
  $f_m^{-1}(s_m)\subset X_m$ with induced morphism $\widehat f _m:\widehat X_m\to
  \widehat S_m$.  We have $\widehat S_m= L_m[[t]]$ and arguing as in Corollary
  \ref{cor:stable-red-one} there is a finite cover $\widehat S'\to \widehat S=k[[t]]$
  and a family of semi-log canonical models $(\widehat X',\widehat B')\to \widehat
  S'$ which over the generic fiber is induced by $([\widehat X_{m,\eta}],[(\widehat
  B_{m,\eta})_{\rm red}])$.  Since $\widehat S'$ is a normal complete 1-dimensional
  DVR, we may assume that $\widehat S'\simeq k[[s]]$ and $\sigma:\widehat S'\to
  \widehat S$ is induced by the inclusion $\sigma ^* :k[[t]]\to k[[s]]$. Let $\sigma
  ^* (t)=s^rg(s)$ where $g(s) \in k[[s]]$ is a unit.  Considering the corresponding
  internal objects $\widehat X'_{\rm int}=[\widehat X'_m]$, $\widehat B'_{\rm
    int}=[\widehat B'_m]$ and $\widehat S'_{\rm int}=[\widehat S_m]$, then $\widehat
  S'_m=L_m[[s]] \to \widehat S_m=L_m[[t]]$ where $t=s^rg(s)$.

  \begin{claim}
    Let $S'_m \to S_m$ be a finite cover ramified at $s_m$ to order $r$ and $s'_m$ be
    the corresponding closed point on $S'_m$.  Then $\widehat S'_m$ is isomorphic to
    the completion of $S'_m$ along $s'_m$.
  \end{claim}
  
  \begin{proof} 
    Let $\gamma :\wt S'_m\to S_m$ be the morphism induced by the above finite cover
    where $\wt S'_m$ is the completion of $S'_m$ along $s'_m$, then $\wt S'_m={\rm
      Spec} (L_m[[s]])$ and $\gamma $ is determined by $\gamma ^* (t)=s^rh_m(s)$
    where we view $t\in R_m$ a local parameter of $S_m$ at $s_m$.  Let $g_m(s)\in
    L_m[[s]]$ be the elements corresponding to $g(s)\in k[[s]]$ and $\alpha_m
    (s),\beta _m(s)\in k[[s]]$ such that $(\alpha_m (s))^r=h_m(s)$ and $(\beta
    _m(s))^r=g_m(s)$, then $\alpha_m (s),\beta _m(s)\in k[[s]]$ are units.  Let $\tau
    _m:L_m[[s]]\to L_m[[s]]$ be a isomorphism such that $\tau _m(\alpha_m(s))=\beta_m(s)$,
    then $\tau _m$ induces the required isomorphism $\widehat S'_m\to \wt S'_m$.

  \end{proof}

  Consider now $\widetilde f _m:\widetilde X_m\to S'_m$ a log resolution of
  $X_m\times _{S_m}S'_m$ such that if $\widetilde B _m$ is the strict transform of
  $B_m$ plus the reduced exceptional divisor and the reduced fiber $(\widetilde f _m ^{-1}(s'_m))_{\rm red}$, then
  $\widetilde B _m$ has simple normal crossings support.  Let $\widehat {\widetilde
    X}_m$ be the completion of $\widetilde X _m$ along $\widetilde f
  ^{-1}_m(s'_m)$. Consider a common resolution $\pi_m:W_m\to \widehat {\widetilde
    X}_m$ and $\varrho_m:W_m\to \widehat X '_m$. We write
  \[
  \pi_m^*(K_{\widehat{\widetilde X}_m}+\widehat{\widetilde B}
  _m)=\varrho_m^*(K_{\widehat X'_m}+\widehat B'_m+\widehat X'_{m,s'_m})+G_m.
  \] 
  It is easy to see that $\pi_{m,*}(G_m)\geq 0$ and since $\varrho_m^*(K_{\widehat
    X'_m}+\widehat B'_m)-\pi_m^*(K_{\widehat{\widetilde X}_m}+\widehat{\widetilde B}
  _m)$ is $\pi _m$-nef, then by the negativity lemma, $G_m\geq 0$ (notice that restricting to the central fiber and applying the usual negativity lemma, we obtain that $G_m|_{(W_m)_{s'_m}}\geq 0$, and  hence $G_m\geq 0$
    as we are working over a formal neighborhood of $s'_m$).

  \begin{claim} 
    The ring $\bigoplus _{q\geq 0}\widetilde f _{m,*}\mathscr O _{\widetilde
      X_m}(q(K_{\widetilde X_m}+\widetilde B _m))$ is a finitely generated $\mathscr
    O _{S'_m}$-algebra.
  \end{claim}

  \begin{proof} Since the statement is triviallly true over the open subset
    $S'_m\setminus \{s'_m\}$, we may localize $S'_m$ at $s'_m$.  Since the natural
    functor from coherent sheaves over $S'_m$ to coherent sheaves over $\widehat S '
    _m$ is exact, it suffices to check the analogous statement for the $\mathscr O
    _{\widehat{S}'_m}$-algebra
    \[
    \bigoplus _{q\geq 0}\widehat{\widetilde f }_{m,*}\mathscr O _{\widehat{\widetilde
        X}_m}(q(K_{\widehat{\widetilde X}_m}+\widehat{\widetilde B} _m)).
    \] 
    By what we have observed above, this algebra is isomorphic to
    \[
    \bigoplus _{q\geq 0}\widehat{ f }'_{m,*}\mathscr O
    _{\widehat{X}'_m}(q(K_{\widehat{ X}'_m}+\widehat{ B}' _m+\widehat X
    _{m,s'_m}))\simeq \bigoplus _{q\geq 0}\widehat{ f }'_{m,*}\mathscr O
    _{\widehat{X}'_m}(q(K_{\widehat{ X}'_m}+\widehat{ B}' _m))
    \] 
    which is finitely generated (since $K_{\widehat{ X}'_m}+\widehat{ B}' _m$ is
    ample over $\widehat S'_m$.
  \end{proof}

  We now consider $X'_m={\rm Proj}(\bigoplus _{q\geq 0}\widetilde f _{m,*}\mathscr O
  _{\widetilde X_m}(q(K_{\widetilde X_m}+\widetilde B _m)))$.  Note that the special
  fiber $X'_{m,s'_m}$ is isomorphic to the special fiber of $\widehat X'_m\to
  \widehat S'_m$ and in particular it is reduced.  By construction $X'_m\to S'_m$ is
  a family of log canonical models
  and these log canonical models determine semi-log canonical models over
  $S'_m\setminus \{ s'_m \}$ and over a formal neigborhood of $s'_m\in S'_m$. Since
  these semi-log canonical models agree over the generic point of the formal
  neigborhood of $s'_m\in S'_m$, we obtain a semi-log canonical model over the whole
  of $S'_m$ (which is automatically projective over $S'_m$ since the relative log
  canonical divisor is relatively ample).  This is the required contradiction and the
  proof is complete.
\end{proof}

\end{subsection}

\end{document}